\documentclass[11pt]{amsart}

\usepackage{bm}
\usepackage{fullpage}
\usepackage{amssymb}
\usepackage{amsmath,amsfonts,amsthm}
\usepackage{hyperref}
\usepackage{color}
\usepackage{cite}

\newtheorem{theorem}{Theorem}

\newtheorem{question}{Question}

\newtheorem{claim}{Claim}

\newtheorem{corollary}{Corollary}

\newtheorem{remark}{Remark}
\usepackage{tikz}
\usetikzlibrary{shapes.geometric}
\usepackage{pst-poly}
\usepackage{graphicx}
\usepackage{subcaption}
\usetikzlibrary{decorations.markings}
\tikzstyle{vertex}=[circle, draw, inner sep=0pt, minimum size=6pt]
\DeclareMathOperator\G{\mathcal{G}}

\DeclareMathOperator\F{\mathcal{F}}

\DeclareMathOperator\C{\mathcal{C}}
\DeclareMathOperator\PP{\mathcal{P}}

\def\isdef{\mbox {$\ \stackrel{\rm def}{=} \ $}}

\title{   partitioning of a graph  into induced subgraphs not containing prescribed cliques}

\author{Yaser Rowshan and Ali Taherkhani}
\address{Y. Rowshan, 
	Department of Mathematics, Institute for Advanced Studies in Basic Sciences (IASBS), Zanjan 45137-66731, Iran}
\email{y.rowshan@iasbs.ac.ir, y.rowshan.math@gmail.com}
\address{A. Taherkhani, 
	Department of Mathematics, Institute for Advanced Studies in Basic Sciences (IASBS), Zanjan 45137-66731, Iran}
\email{ali.taherkhani@iasbs.ac.ir}
\begin{document}
	\maketitle 
	
	\begin{abstract} 
		
Let  $K_p$ be a complete  graph of order  $p\geq 2$. A $K_p$-free $k$-coloring of a graph $H$ is a partition of $V(H)$ into
$V_1, V_2\ldots,V_k$ such that  $H[V_i]$ does not contain $K_p$ for each $i\leq k $. 
In 1977 Borodin and Kostochka  conjectured that any graph $H$ with maximum degree $\Delta(H)\geq 9$ and without $K_{\Delta(H)}$ as a 
subgraph has chromatic number at most $\Delta(H)-1$.
As analogue of the Borodin-Kostochka conjecture, 
we prove that if  $p_1\geq \cdots\geq p_k\geq 2$, $p_1+p_2\geq 7$, $\sum_{i=1}^kp_i=\Delta(H)-1+k$, and $H$ does not contain $K_{\Delta(H)}$   as a subgraph, 
then there is a partition of $V(H)$ into $V_1,\ldots,V_k$ such that for each $i$, $H[V_i]$ does not contain $K_{p_i}$.
In particular, if $p\geq 4$ and $H$ does not contain $K_{\Delta(H)}$   as a subgraph,
then $H$ admits a $K_p$-free $\lceil{\Delta(H)-1\over p-1}\rceil$-coloring. 
Catlin showed that every connected non-complete graph $H$ with
$\Delta(H)\geq 3$ has a $\Delta(H)$-coloring such that one of  the color classes is maximum $K_2$-free subset (maximum independent set).
In this regard, we show that there is a partition of vertices of $H$ into $V_1$ and $V_2$ such that 
$H[V_1]$ does not contain $K_{p}$, $H[V_2]$ does not contain $K_{q}$, and $V_1$ is a maximum $K_p$-free subset of V(H)
if   $p\geq 4$,  $q\geq 3$, $p+q=\Delta(H)+1$, and its clique number  $\omega(H)=p$.
	\end{abstract}
	
	\section{Introduction} 
	All graphs considered in this paper are undirected, simple, and finite. For  a graph  $H=(V(H),E(H))$,  
	 the neighborhood of $v$, denoted by  $N_H(v)$ (or for simply $N(v)$), is  the set of neighbors of $v$ in $H$.
	 Also, the closed neighborhood of $v$, denoted by $N[v]$, is defined as
$N[v] = N(v) \cup \{v\}$.
	Let   $\deg_H{(v)}$ (or for simply $\deg{(v)}$) be $|N_H(v)|$.
	We denote the  maximum degree and minimum degree of $H$ by $\Delta(H)$ and $\delta(H)$, respectively.   
	Suppose that  $W$ is a subset of $V(H)$. The induced subgraph $H[W]$ is the graph 
	whose vertex set is $W$ and whose edge set consists of all of the edges in $E(H)$ that have both endpoints in $W$.

	We recall that an independent set  in  a graph $H$ is a subset of its vertices that does not contain any  edge of $H$. 
	The independence number of a graph $H$ is the size of a maximum independent set and is denoted by $\alpha(H)$. 
	A  clique in a graph $H$ is a set of pairwise adjacent vertices. 
 The clique number of a graph $H$, written $\omega(H)$, is  the maximum size of a clique of a graph $H$.

 	The strong product of two graphs $G$ and $H$, denoted by $G \boxtimes H$, is a graph whose vertex set is the Cartesian product $V(G) \times V(H)$, and two vertices
	 $(u, v)$ and $(u', v')$ are adjacent in $G \boxtimes H$ whenever 
$uu'\in E(G)$ and $v=v'$, or $vv'\in E(H)$ and $u=u'$, or $uu'\in E(G)$ and $vv'\in E(H)$.
	
	For a positive integer $k$, a $k$-coloring of $H$ is a partition of $V(H)$ into subsets $V_1,\ldots,V_k$ such that each $V_i$ is
	an independent set. The chromatic number of $H$, denoted by $\chi(H)$, is the smallest $k$ for which $H$ has a $k$-coloring. 
	The first nontrivial upper bound on chromatic number is shown by Brooks in 1941 \cite{Brooks}. It states that for any graph $G$  with 
	maximum degree $\Delta(G)$, clique number $\omega(G)$,
	and  chromatic number $\chi(G)$ we have  $\chi(G)\leq \max\{\Delta(G),\omega(G)\}$.
	One of the well-known conjectures in the area of graph colorings  is the Borodin and Kostochka conjecture. 
	As an extension of the Brooks theorem,
	in 1977 Borodin and Kostochka  conjectured that any graph $G$ with maximum degree $\Delta(G)\geq 9$ and clique number $\omega(G)\leq\Delta(G)-1$, 
	   satisfies 
	$\chi(G)\leq \Delta(G)-1$.
	Reed proved that  the conjecture holds whenever $ \Delta(G) \geq 10^{14} $ \cite{REED1999}.
	
	When $\chi(H)=\Delta(H)$ there are a series of interesting and useful results shows that $\omega(H)$
	must be close to $\Delta(H).$
	As the first result Borodin and Kostochka showed if $\chi(H) =\Delta(H) \geq 7$, then $H$ contains $K_{\frac{\Delta(H)+1}{2}}$ \cite{BORODIN1977247}. 
	Then,  Kostochka showed $H$ must  contain a copy of $K_{\Delta(H)-28}$\cite{kostochka1980degree}. 
	Mozhan proved that $\omega(H)\geq \Delta(H)-3$ when $\Delta(H)\geq 31$\cite{mozhan}. Finally, Cranston and Rabern strengthened Mozhan's 
	result   by weakening the condition to $\Delta(H)\geq 13$ \cite{cranston2015}.

	Let ${\PP}$ be  a graphical property.
	The conditional chromatic number $\chi(H,\PP)$ of $H$,  is the smallest  integer  $k$ for which there 
	is a decomposition of  $V(H)$  into  sets $V_1,\ldots,V_k$ such that for each $1\leq i\leq k$, $H[V_i]$ 
	satisfies  property $\PP$.
	Harary in 1985 introduced this extension of graph coloring ~\cite{MR778402}.
	Suppose that $\G$ is a family of graphs. For the property $\PP$ of being $\G$-free (i.e. a graph that has this property does not contain any subgraph isomorphic to a graph from $\G$),  we write $\chi_{\G}(H)$ instead of $\chi(H, \PP)$. Here, we say  a graph $H$ has a $\G$-free $k$-coloring if 
	$\chi_{\G}(H)\leq k$. For simplicity of notation, when $\G=\{G\}$, we write $\chi_G(H)$ instead of $\chi_{\G}(H)$. 
	An ordinary  $k$-coloring of $H$  can be viewed as a $K_2$-free $k$-coloring of a graph $H$.  
	Let the family $\C$ consist of all cycles.
	A $\C$-free $k$-coloring of a graph $H$ is a partition of $V(H)$ into $V_1,\ldots,V_k$ such that each $H[V_i]$ is 
	acyclic. The vertex arboricity of a graph $H$,  denoted by $a(H)$,
	is the minimum $k$ for which $H$ has a $\C$-free $k$-coloring. 
	 The vertex arboricity was introduced by  Chartrand, Kronk, and Wall in~\cite{Chartrand}.
	
	As a hardness result for $G$-free coloring,  Alchioptas showed that if $|V(G)|\geq 3$,
	verifying whether a graph admits a $2$-coloring such that each color class does not contain any copy of $G$ 
	is NP-complete while it is well-known that the property of being 2-colorable can be solved in linear time\cite{achlioptas1997complexity}.
	
	As a generalization of Brooks' Theorem, Catlin showed that every graph $H$ with $\Delta(H)\geq 3$
	without $K_{\Delta(H)+1}$ as a subgraph, has a $\Delta(H)$-coloring such that one of  the color classes is 
	a maximum independent set~\cite{Catlin}.  Also,  for the vertex arboricity of $H$ Catlin and Lai proved a result similar to  Catlin's result
	 by using $\lceil \frac{\Delta(H)}{2}\rceil$ colors~\cite{Catlin1}.
	%%%%%%%%%%%%%%%%%%%%%%%%%%%%%%%%%%%%%%%%%%%
	%%%%%%%%%%%%%%%%%%%%%%%%%%%%%%%%%%%%%%%%%%%
	%%%%%%%%%%%%%%%%%%%%%%%%%%%%%%%%%%%%%%%%%%%
	%%%%%%%%%%%%%%%%%%%%%%%%%%%%%%%%%%%%%%%%%%%
	%%%%%%%%%%%%%%%%%%%%%%%%%%%%%%%%%%%%%%%%%%%
	%%%%%%%%%%%%%%%%%%%%%%%%%%%%%%%%%%%%%%%%%%%
	%%%%%%%%%%%%%%%%%%%%%%%%%%%%%%%%%%%%%%%%%%%
	%%%%%%%%%%%%%%%%%%%%%%%%%%%%%%%%%%%%%%%%%%%
	%%%%%%%%%%%%%%%%%%%%%%%%%%%%%%%%%%%%%%%%%%%
	In~\cite{rowshan2020catlin} the authors showed that a Catlin-type result  holds for $G$-free coloring. Let $G_1,\ldots, G_{k}$
	be  $k$ connected graphs with minimum degrees  $d_1,\ldots,d_k$, respectively, and let $H$ be a graph with $\Delta(H)=\sum_{i=1}^k d_i$. 
	Extending Catlin’s result, it was shown that there exists a 
	 a partition of  $V(H)$ into $V_1,\ldots, V_k$
	such that every $H[V_i]$ is $G_i$-free  
	except  either
	\begin{itemize}
	\item $k=1$ and $H$ is  isomorphic to  $G_1$, 
	\item each
	$G_i$ is isomorphic to $K_{d_i+1}$ and $H$ is not isomorphic to $K_{\Delta(H)+1}$, or
	\item	each $G_i$ is isomorphic to $K_{2}$ and $H$ is not an odd cycle.
	 	\end{itemize}
Furthermore, 
	one of  the $V_i$'s  can be chosen in a way that $H[V_i]$  is a maximum   $G_i$-free subset of $V(H)$.

\noindent One can ask the following natural question which is an analogue of Borodin and Kostochka's conjecture.
	
	%\begin{question}
	%Let $H$ be a graph with maximum degree $\Delta(H)\geq 5$ and clique number $\omega(G)$ where $\omega(G)\leq \Delta(G)-1$. 
	%Assume  that $p$ and $q$ are two positive integers and $p+q=\Delta(H)+1$. Is there a partition of vertices of $H$ into $V_1,V_2$
	%such that $H[V_1]$ and $H[V_2]$ are $K_p$-free and $K_q$-free, respectively?
	%\end{question}
	\begin{question}~\label{genBK}
		Suppose that $H$ is a graph with $\Delta(H)\geq 2$ and clique number $\omega(H)$ where $\omega(H)\leq \Delta(H)-1$. 
		Assume  that  $p_1\geq p_2\geq\cdots\geq p_k\geq 2$  are  $k$ positive integers 
		and  $\sum_{i=1}^k p_i=\Delta(H)-1+k$. Is there a partition of  $V(H)$ into $V_1,V_2,\ldots, V_k$
		such that for each $1\leq i\leq k$, $H[V_i]$ is $K_{p_i}$-free?
	\end{question}
	In the case $\Delta(H)=2$,  we have  $k+1=\sum_{i=1}^k p_i\geq 2k$. Therefore, we obtain $k=1$ and  $p_1=2$. As a consequence the answer of 
	Question~\ref{genBK} is negative. 
So we may assume that $\Delta(H)\geq 3$. 
	Note that the answer to this question for the case $k=1$ is positive since in this case we have $p_1=\Delta(H)$  and then $\omega(H)\leq \Delta(H)-1=p_1-1$. So assume that $k\geq 2$.
	Also, note that  the above question the case when  $p_1=2$ (i.e. all $p_i$'s are equal to $2$) and $\Delta(H)\geq 9$ is the Borodin and Kostochka conjecture \cite{BORODIN1977247}. 	The case when  $p_1=2$ (i.e. all $p_i$'s are equal to $2$)  and $\Delta(H)\leq 8$, the answer of the above question is negative. 
	For example  for the $\Delta(H)=5$ case, take the strong product of $C_5$ and $K_2$,
	and  for the $\Delta(H)=8$ case, take the strong product of $C_5$ and $K_3$.
	 We intend to study  Question~\ref{genBK} when $p_1\geq 3$.

	 In the case $\Delta(H)=3$, since $p_i$'s are at least $2$ we must have either $k=1$ and $p_1=3$ or $k=2$ and $p_1=p_2=2$. 
	 In this case when $k=2$ and $p_1=p_2=2$,  every graph $H$
	  with maximum degree $\Delta(H)=3$ and chromatic number $\chi(H)=3$  is a negative answer to Question~\ref{genBK}.
	  
	 In the next remark, let us check what happens when $4\leq \Delta(H)\leq 6$ and $p_1\geq 3$ in Question~\ref{genBK}.
	 \begin{remark}\label{rmk1}
	Note that under the assumptions of Question~\ref{genBK}, if we have $4\leq \Delta(H)\leq 6$, $p_1\geq 3$, then  since $\Delta(H)-1+k=\sum_{i=1}^{k}p_i$  only the following cases can occur, and for some of these cases we can provide a positive or negative answer.
	\begin{itemize}
		\item[(a)] If $\Delta(H)=4$, $k=2$, $p_1=3$, and  $p_2=2$, then  the graph $H_1$  shown in Figure \ref{fi0} gives a negative answer to 
		Question~\ref{genBK}.
		\item[(b)]  If $\Delta(H)=5$, $k=2$, $p_1=4$, and  $p_2=2$, then   the strong product of $C_{2t+1}$ and $K_2$, where $t\geq 2$,  is a negative answer to Question~\ref{genBK}.
		\item[(c)]  If $\Delta(H)=5$,  $k=3$,  $p_1=3$, and $p_2=p_3=2$, then  the strong product of $C_{2t+1}$ and $K_2$,  where $t\geq 2$,  is a negative answer to 
		Question~\ref{genBK}.
		\item[(d)]  If $\Delta(H)=5$, $k=2$, and $p_1=p_2=3$,  we do not know the answer.
		\item[(e)]  If $\Delta(H)=6$, $k=4$, $p_1=3$, and $p_2=p_3=p_4=2$, we do not know the answer.
		\item[(f)] If $\Delta(H)=6$, $k=3$, $p_1=p_2=3$, and $p_3=2$, we do not know the answer.
		\item[(g)]  If $\Delta(H)=6$, $k=3$, $p_1=4$, and $p_2=p_3=2$, we do not know the answer.
		\item[(h)]  If $\Delta(H)=6$, $k=2$, $p_1=4$, and $p_2=3$, the answer is positive, it comes from Theorem~\ref{mainthm}.
		\item[(i)]  If $\Delta(H)=6$, $k=2$, $p_1=5$, and $p_2=2$, the answer is positive, it comes from Theorem~\ref{mainthm}.
		 
			\end{itemize}
	  		\end{remark}

	  	We will provide an explanation why the Parts (a), (b), and (c) do not have the required partitions at the beginning of the next section.

	%%%%%%%%%%%%%%%%%%%%%%%%%%%%%%%%%%%%%%%%%
		
	%%%%%%%%%%%%%%%%%%%%%%%%%%%%%%%%%%%%%%%%%
	%%%%%%%%%%%%%%%%%%%%%%%%%%%%%%%%%%%%%%%%%
	%%%%%%%%%%%%%%%%%%%%%%%%%%%%%%%%%%%%%%%%%
	%%%%%%%%%%%%%%%%%%%%%%%%%%%%%%%%%%%%%%%%%
	%%%%%%%%%%%%%%%%%%%%%%%%%%%%%%%%%%%%%%%%%
	%%%%%%%%%%%%%%%%%%%%%%%%%%%%%%%%%%%%%%%%%
	%%%%%%%%%%%%%%%%%%%%%%%%%%%%%%%%%%%%%%%%%
	%%%%%%%%%%%%%%%%%%%%%%%%%%%%%%%%%%%%%%%%%
	%%%%%%%%%%%%%%%%%%%%%%%%%%%%%%%%%%%%%%%%%
	%%%%%%%%%%%%%%%%%%%%%%%%%%%%%%%%%%%%%%%%%
	%%%%%%%%%%%%%%%%%%%%%%%%%%%%%%%%%%%%%%%%%
	%%%%%%%%%%%%%%%%%%%%%%%%%%%%%%%%%%%%%%%%%
	%%%%%%%%%%%%%%%%%%%%%%%%%%%%%%%%%%%%%%%%%
	In the next theorem we show that if  in Question~\ref{genBK} we have $p_1+p_2\geq 7$, the answer to this question  is positive.

It is worth noting that if  the Borodin and Kostochka conjecture holds and in Question~\ref{genBK}  we have $p_1\geq 3$ and $\Delta(H)\geq 9$, then the answer of Question~\ref{genBK} is 
positive.	It follows from the fact that  if  $W$  is a subset of  vertices of $H$ and there is a partition of $W$ into $p-1$ $K_2$-free subsets $W_1,\ldots,W_{p-1}$, then
	$H[W]$ is $K_p$-free.

	\begin{theorem}~\label{mainthm}
		Suppose that $H$ is a graph  with $\omega(H)\leq \Delta(H)-1$. Let $k\geq 2$ be a positive integer.
		Assume  that  $p_1\geq p_2\geq\cdots\geq p_k\geq 2$  are  $k$ positive integers and  $\sum_{i=1}^k p_i=\Delta(H)-1+k$. If $p_1+p_2\geq 7$, then
		there exists a partition of $V(H)$ into $V_1,V_2,\ldots, V_k$
		such that for each $1\leq i\leq k$, $H[V_i]$ is $K_{p_i}$-free.
	\end{theorem}
	Note that  under assumptions of  Theorem~\ref{mainthm} we conclude that $\Delta(H)\geq 6$. For see this,  since $p_1+p_2\geq 7$,  all $p_i$'s are  least $2$,  and  $\sum_{i=1}^k p_i=\Delta(H)-1+k$,
	we have $7+2(k-2)\leq \Delta(H)-1+k$ and consequently $\Delta(H)\geq k+4\geq 6$.
	
	As a direct consequence of Theorem~\ref{mainthm}, one can see that if   $H$ is a graph with $\omega(H)\leq \Delta(H)-1$  and $p\geq 4$, then
	$\chi_{_{K_p}}(H)\leq \lceil{\Delta(H)-1\over p-1}\rceil$.
	
	%One can check that by Theorem~\ref{King2/3} instead of Theorem~\ref{mainthm}, we can proof the following result.
	%\begin{theorem}\label{mth2}
	%	Assume that $H$  is a  graph with maximum degree $\Delta(H)\geq 6$ and $\omega(H)= \omega$, where $\omega(H)\leq \Delta(H)-1$.
	%	Let $p, q$ be two positive integers, where $p\geq q\geq 2$ and  $p+q=\Delta(H)+1$.  If $p\geq 4$, then there exists a partition of $V(H)$ into  $V_1, V_2$ 
	%	such that $H[V_1]$ is $K_p$-free and  $H[ V_2] $ is $K_q$-free.
	
	%\end{theorem}
	
	%%%%%%%%%%%%%%%%%%%%%%%%%%%%%%%%%%%%%%%%%
	%%%%%%%%%%%%%%%%%%%%%%%%%%%%%%%%%%%%%%%%%
	%%%%%%%%%%%%%%%%%%%%%%%%%%%%%%%%%%%%%%%%%
	%%%%%%%%%%%%%%%%%%%%%%%%%%%%%%%%%%%%%%%%%
	%%%%%%%%%%%%%%%%%%%%%%%%%%%%%%%%%%%%%%%%%
	%%%%%%%%%%%%%%%%%%%%%%%%%%%%%%%%%%%%%%%%%
	%%%%%%%%%%%%%%%%%%%%%%%%%%%%%%%%%%%%%%%%%
	%%%%%%%%%%%%%%%%%%%%%%%%%%%%%%%%%%%%%%%%%
	%%%%%%%%%%%%%%%%%%%%%%%%%%%%%%%%%%%%%%%%%
	
	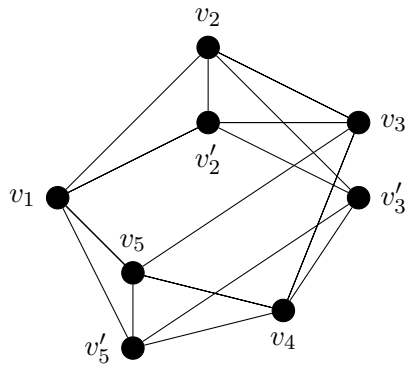
\begin{figure}[ht]
		\begin{tabular}{ccc}
			\begin{tikzpicture}
				\node [draw, circle, fill=black, inner sep=3pt, label=left: $v_1$ ] (y1) at (0,2) {};
				\node [draw, circle, fill=black, inner sep=3pt, label=above: $v_2$] (y2) at (2,4) {};
				\node [draw, circle, fill=black, inner sep=3pt, label=below:$v'_2$] (y3) at (2,3) {};
				\node [draw, circle, fill=black, inner sep=3pt, label=right:$v_3$] (y4) at (4,3) {};
				\node [draw, circle, fill=black, inner sep=3pt, label=right:$v'_3$] (y5) at (4,2) {};
				\node [draw, circle, fill=black, inner sep=3pt, label=above:$v_5$] (y6) at (1,1) {};
				\node [draw, circle, fill=black, inner sep=3pt, label=below:$v_4$] (y7) at (3,.5) {};
				\node [draw, circle, fill=black, inner sep=3pt, label=left:$v'_5$] (y8) at (1,0) {};
				\
				
				\draw (y1)--(y3)--(y1)--(y2);
				\draw (y2)--(y3);
				\draw (y2)--(y4)--(y2)--(y5);
				\draw (y4)--(y6);
				\draw (y5)--(y8);
				\draw (y3)--(y4)--(y3)--(y5);
				\draw (y1)--(y6)--(y1)--(y8);
				\draw (y6)--(y8);
				\draw (y7)--(y6)--(y7)--(y8);
				\draw (y7)--(y4)--(y7)--(y5);

			\end{tikzpicture}
		\end{tabular}\\
		\caption{The graph $H_1$, a negative answer to Question~\ref{genBK} for $\Delta=4$, $k=2$, $p_1=3$, and $p_2=2$.}
		\label{fi0}
	\end{figure}
	%============================================
	
	As we mentioned  Catlin proved that 
	if $H$ is a graph
	with $\Delta(H)\geq 3$ and 
	without $K_{\Delta(H)+1}$ as a subgraph, then $H$ has a proper $\Delta(H)$-coloring for which one of  the color classes 
	is a maximum independent set of $H$~\cite{Catlin}. One can ask for a given graph $H$ with the chromatic number at most $\Delta(H)-1$, 
	is there any  $(\Delta(H)-1)$-coloring of $H$ such that one of  the color classes  is  a maximum independent set of $H$? 
	It is easy to see that the answer is negative.
	For example for $n\geq 2$ consider the complete $K_n$ and for each vertex of $K_n$, say $v_i$, add two new vertices $w_i, w'_i$ and join them to $v_i$.
	The resulting graph $H_0$ has  $\Delta(H_0)=n+1$, $\chi(H_0)=n$, and $\alpha(G)=2n$. 
	The graph $H_0$ has a unique independent set of cardinality $2n$,
	which can not be a color class for any  $(\Delta(H_0)-1)$-coloring of $H_0$.

		There is another question that is closely related  to Catlin's result is the following. 

%	\begin{question}\label{Catlinqu}
%		Let $H$ be a graph with maximum degree $\Delta(H)\geq 5$ and clique number $\omega(H)$ where $\omega(H)\leq \Delta(H)-1$. 
%		Assume  that  $p_1\geq p_2\geq\cdots\geq p_k\geq 2$  are  $k$ positive integers
%		and  $\sum_{i=1}^k p_i=\Delta(H)-1+k$. If $p_1\geq 3$, then
%		there exists a partition of vertices of $H$ into $V_1,V_2,\ldots, V_k$
%		such that $H[V_i]$ is $K_{p_i}$-free for each $1\leq i\leq k$ and $H[V_1]$ is a maximum $K_{p_1}$-free subgraph.
%	\end{question}
	
%	Another version of the Question~\ref{Catlinqu} can be asked

	\begin{question}\label{Catlinqu}
		Let $H$ be a graph. 
		Assume  that  $p_1\geq p_2\geq\cdots\geq p_k\geq 2$  are  $k$ positive integers and  $\sum_{i=1}^k p_i=\Delta(H)-1+k$. 
		What is the  minimum number $f(\Delta, p_1)$ such that the following is true?
		If $H$ is a graph with  $\omega(H)\leq \Delta(H)-f(\Delta,p_1) $, then 
		there exists a partition of vertices of $H$ into $V_1,V_2,\ldots, V_k$
		such that $H[V_i]$ is $K_{p_i}$-free for each $1\leq i\leq k$ and $H[V_1]$ is a maximum $K_{p_1}$-free subgraph.
	\end{question}
	Note that in Question~\ref{Catlinqu}, if $H$ contains a copy of $K_{\Delta(H)}$,  such a partition does not exist. Therefore, 
	we must have $\omega(H)\leq \Delta(H)-1$ and thus, $ f(\Delta, p_1)\geq 1$.
	If $\omega(H)\leq p_1-1$, $H$  contain no copy of $K_{p_1}$. Therefore, 
	 $1\leq f(\Delta, p_1)\leq \Delta(H)-(p_1-1)$.
	 Here, we prove that if $k=2$ and $p_1\geq 4$ and $p_2\geq 3$, then  $f(\Delta, p_1)\leq \Delta(H)-p_1$. 
		
	\begin{theorem}\label{mth1}
		Assume that $H$  is a  graph with  $\Delta(H)\geq 6$ and  clique number $\omega(H)$ where $4\leq \omega(H)\leq \Delta(H)-2$.
		Denote  $\omega(H)=p$ and  $\Delta(H)+1-p=q$.
				Then there exists  $V_1\subseteq V(H)$ such that $V_1$ is a maximum $K_p$-free subset of $H$,
		and $H[V\setminus V_1] $ is $K_q$-free.	
	\end{theorem}

		As a direct  consequence of   Theorems~\ref{mainthm} and \ref{Catlinqu} we can show that  if $k\geq 3$,
	$p_1\geq 4$, and $p_2+p_3\geq 7$, then  $f(\Delta, p_1)\leq \Delta(H)-p_1$. We defer the proof of the following corollary to the end of  Section~\ref{s2}.

\begin{corollary}\label{cor1}
Let $H$ be a graph. 
		Assume  that  $k\geq 3$ and  $p_1\geq p_2\geq\cdots\geq p_k\geq 2$  are  $k$ positive integers for which
		$p_1\geq 4$ and  $p_2+p_3\geq 7$, and  $\sum_{i=1}^k p_i=\Delta(H)-1+k$. 
		If $\omega(H)=p_1$, then there exists a partition of vertices of $H$ into $V_1,V_2,\ldots, V_k$
		such that $H[V_i]$ is $K_{p_i}$-free for each $1\leq i\leq k$. Moreover,  $H[V_1]$ is a maximum $K_{p_1}$-free subgraph.
\end{corollary}

	\section{Proof of the main results}\label{s2}
	In this part we shall prove  Parts (a), (b), and (c) of  Remark~\ref{rmk1}, Theorems \ref{mainthm} and   \ref{mth1}, and Corollary~\ref{cor1}.

	%%%%%%%%%%%%%%%%%%%%%%%%%%%%%%%%%%%%%%%%%
	%%%%%%%%%%%%%%%%%%%%%%%%%%%%%%%%%%%%%%%%%
	%%%%%%%%%%%%%%%%%%%%%%%%%%%%%%%%%%%%%%%%%
	%%%%%%%%%%%%%%%%%%%%%%%%%%%%%%%%%%%%%%%%%
	%%%%%%%%%%%%%%%%%%%%%%%%%%%%%%%%%%%%%%%%%
	%%%%%%%%%%%%%%%%%%%%%%%%%%%%%%%%%%%%%%%%%
	%%%%%%%%%%%%%%%%%%%%%%%%%%%%%%%%%%%%%%%%%
	%%%%%%%%%%%%%%%%%%%%%%%%%%%%%%%%%%%%%%%%%
	%%%%%%%%%%%%%%%%%%%%%%%%%%%%%%%%%%%%%%%%%
\begin{proof}[\bf Proof of  Parts (a), (b), and (c) of Remark~\ref{rmk1}] 

\item To prove Part~(a), consider the graph $H_1$ in Figure~\ref{fi0}. Our aim is  that to prove the vertex set of  $H_1$ cannot be partitioned into a $K_2$-free subset
  and a $K_3$-free subset. By way of contradiction, suppose that we have an independent set $I$ such that $H\setminus I$ is $K_3$-free.
  Consider two vertex disjoint  copies of $K_3$ in $H_1$, namely $T_1\isdef H_1[\{v_1,v_2,v'_2\}]$ and $T_2\isdef H_1[\{v_4,v_5,v'_5\}]$.
  Then, the  independent set $I$ must contain one vertex of $T_1$ and one vertex of $T_2$. If we have $v_1\in I\cap T_1$, then from $T_2$ we must have
  $v_4\in I$ since $v_1$ is adjacent to $v_5$ and $v'_5$. Then, $H_1[v_2,v'_2,v_3]$ is a copy of $K_3$ in $H\setminus I$, a contradiction.
  Therefore, we may assume that $v_1\notin I$. By symmetry assume that from $T_1$ the vertex $v_2\in I\cap T_1$. 
  Since $T_3\isdef H_1[\{v_1,v_5,v'_5\}]$ is a copy of $K_3$, again by symmetry we may assume that $v_5$ from $T_3$ must be in $I$.
  Then, $H_1[v'_3,v_4,v'_5]$ is a copy of $K_3$ in $H\setminus I$, a contradiction.

  \item To prove Parts~(b) and (c), for a fixed $t\geq 2$ consider the strong product of $C_{2t+1}$ and $K_2$, which we will denote by $H$ here.
  To prove (b),
   our aim is that that the vertex set of  $H$ 
  cannot be partitioned into a $K_2$-free subset
  and a $K_4$-free subset. By way of contradiction, suppose that we have an independent set $I$ such that $H\setminus I$ is $K_4$-free.
   Since $\alpha(H)=t$, we have $|I|\leq t$. 
  The graph  $H$ has $2t+1$ distinct copy of $K_4$ and each vertex of $H$ lies in exactly two copies of $K_4$.
  Therefore, $I$ intersects at most $2t$ copies of $K_4$. Then $H\setminus I$ contains a copy of $K_4$, a contradiction.
   
\item  To prove Part~(c), our aim is that  the vertex set of  $H$ 
  cannot be partitioned into two $K_2$-free subsets 
  and a $K_3$-free subset.  By way of contradiction, suppose that  we have two disjoint independent sets $I$ and $I'$ in $H$ such that 
  $H\setminus (I\cup I')$ is $K_3$-free.
   Since $\alpha(H)=t$, we have $|I|\leq t$ and $|I'|\leq t$. 
 By using a similar argument to that in Part~(b), we have $H\setminus I$ contains a copy of $K_4$. The independent set $I'$ has at most one vertex in this copy of $K_4$ in $H\setminus I$.
 Hence, $H\setminus(I\cup I')$ has a copy of $K_3$, a contradiction.
   \end{proof}

	%%%%%%%%%%%%%%%%%%%%%%%%%%%%%%%%%%%%%%%%%
	%%%%%%%%%%%%%%%%%%%%%%%%%%%%%%%%%%%%%%%%%
	%%%%%%%%%%%%%%%%%%%%%%%%%%%%%%%%%%%%%%%%%
	%%%%%%%%%%%%%%%%%%%%%%%%%%%%%%%%%%%%%%%%%
	%%%%%%%%%%%%%%%%%%%%%%%%%%%%%%%%%%%%%%%%%
	%%%%%%%%%%%%%%%%%%%%%%%%%%%%%%%%%%%%%%%%%
	%%%%%%%%%%%%%%%%%%%%%%%%%%%%%%%%%%%%%%%%%
	%%%%%%%%%%%%%%%%%%%%%%%%%%%%%%%%%%%%%%%%%
%%%%%%%%%%%%%%%%%%%%%%%%%%%%%%%%%%%%%%%%%%%%%%%%%%%%%%%%%%%%%%%%%%%%%%%%%%%%%

	To prove Theorem~\ref{mainthm} we need  the following useful and  interesting  result due to Christofides, Edwards, 
	and King~\cite{christofides2013note}, which is a generalization of  some earlier results  due to
	Koatochka~\cite{kostochka1980degree}, Rabern~\cite{rabern2011hitting}, and King~\cite{king2011hitting}.
	It was shown that graphs with clique number sufficiently close to their maximum degree have an independent set hitting every maximum clique.
	\begin{theorem}{\rm \cite{christofides2013note}\label{M.th2}}
		Any connected graph $H$ satisfying $\omega(H)\geq \frac{2(\Delta(H)+1)}{3}$, contains an independent set  intersecting every maximum clique unless it is the strong product of an odd cycle with length at least $5$ and the complete graph $K_{\omega(H)/ 2}$.
	\end{theorem}

	%%%%%%%%%%%
	%%%%%%%%%%%%%%%%%%%%%%%%%%%%%%%%%%%%%%%%%%%%%%%%%%%%%%%%%%%%%%%%%%%%%%%%%%%%%%%%%%%%%%%%%%%%%%%%%%%%%%%%%%%%%%%%%%%%%%%%%%%%%%%%%%%%%%%%
	%%%%%%%%%%%%%%%%%%%%%%%%%%%%%%%%%%%%%%%%%%%%%%%%%%%%%%%%%%%%%%%%%%%%%%%%%%%%%%%%%%%%%%%%%%%%%%%%%%%%%%%%%%%%%%%%%%%%%%%%%%%%%%%%%%%%%%%%
	\begin{proof}[\bf Proof of  Theorem \ref{mainthm}]
		First note that  we only need to prove the statement for $k=2$. Assume that the statement  holds for $k=2$, i.e. 
		if  $\omega(H)\leq \Delta(H)-1$, $p+q=\Delta(H)+1$,  and $p+q\geq 7$, 
		then there exists a partition of $V(H)$ into  $V_1, V_2$ so that $H[V_1]$ is $K_p$-free and  $H[ V_2] $ is 
		$K_q$-free. 
		
		To prove the statement for $k\geq 3$, we
		set $p=\sum_{i=1}^{k-1}p_i-(k-2)$ and $q=p_k$. Note that $p+q=\Delta(H)+1$ and also   $p+q\geq 7$  
		since $p+q$ is equal to $\sum_{i=1}^{k}p_i-(k-2)\geq p_1+p_2+(k-2)$.
		Since the statement is true for $k=2$, we can obtain 
		a partition of $V(H)$ into  $V_1$ and $V_2$ such that $H[V_1]$ is $K_p$-free and  $H[ V_2] $ is $K_q$-free. Also, assume that 
		$V_2$ is a maximal $K_q$-free subset of $H$. 	Therefore, each vertex in $V_1$ has at least $q-1$ neighbors  in $V_2$,
				 which implies that the maximum degree of $H[V_1]$ is at most $\Delta(H)-(q-1)=p$.
			
			If the maximum degree of $H[V_1]$ is less than $p$, let $v\in V_1$ be a vertex with maximum degree in $H[V_1]$ such that its degree is equal to $p'<p$. We add $p-p'$ new vertices to $H[V_1]$ and join all of them to $v$, forming a new graph $H'$. The graph $H'$ has maximum degree $p$ and is $K_p$-free.	 
				 Suppose that there exists a partition of $V(H')$   into  $W_1,\cdots, W_{k-1}$ such that for each $1\leq i\leq k-1$, $H'[W_i]$ is $K_{p_i}$-free,
				  then $W_1\cap V_{1},\ldots, W_{k-1}\cap V_1, V_2$ is the desired partition of $V(H)$.
				 Therefore, we may assume that $H[V_{1}]$  is a graph with maximum degree $p$. 
				 We also have $\omega( H[V_{1}])\leq p-1$, $p_1+p_2\geq 7$, $\sum_{i=1}^{k-1}p_i=p-1+(k-1)$.
				  We iterate this procedure until we  obtain the desired partition. 
				  
		To prove the statement  of Theorem~\ref{mainthm} for $k=2$,    we use induction on $\Delta(H)$. 
		
		 We may assume that $H$ is a connected graph; otherwise if $H$ has $\ell\ge 2$ connected components,
		say $H_1,\ldots, H_{\ell}$,  we prove the statement for each of these connected components. Then, for each connected component  $H_i$, 
		there exists a partition of $V(H_i)$ into  $V_{i,1}$ and $V_{i,2}$  such that $H[V_{i,1}]$ is $K_p$-free and  $H[ V_{i,2}] $ is $K_q$-free.
		 Now, define $V_1=\cup_{i=1}^{\ell}V_{i,1}$ and
		$V_2=\cup_{i=1}^{\ell}V_{i,2}$. Therefore, we have $H[V_1]$ is $K_p$-free and  $H[ V_2] $ is $K_q$-free.		
		
		For the base case of the induction, let   $\Delta(H)=6$. Since $p+q=\Delta(H)+1=7$ and $p\geq q\geq 2$, we must have either $(p,q)=(4,3)$ or
		$(p,q)=(5,2)$.  As $H$ is connected graph with  $\Delta(H)=6$, $H$ is not isomorphic to
		the strong product of an odd cycle and a complete graph. 
		
		First assume that $(p,q)=(5,2)$. If $\omega(H)\leq 4$, there is nothing to prove. Thus, we may assume  that $\omega(H)=5$. 
		By Theorem~\ref{M.th2},  there exists  an independent set $I\subseteq V(H)$ such that $I$ intersects each maximum clique of $H$.
		Set $V_1=V(H)\setminus I$ and $V_2=I$. Note that $H[V_1]$ is $K_5$-free  
		and $H[V_2]$ is $K_2$-free.
		
		Assume that $(p,q)=(4,3)$. If $\omega(H)\leq 3$, there is nothing to prove. Thus, we may assume that $\omega(H)\in\{4,5\}$. If $\omega(H)=5$, then
		  Theorem~\ref{M.th2} implies that there exists  a maximal 
		independent set $I\subseteq V(H)$  such that $I$ intersects each clique of size $5$ in $H$.	
		If $\omega(H)=4$, then we take  $I$ to be a maximal 
		independent set in $H$.
		Thus, in both cases, the subgraph $H\setminus I$ has maximum degree at most $5$ and 
		clique number at most $4$.  If $\omega(H\setminus I)\leq 3$, then there is nothing to prove. 
		So, we may assume that $\omega(H\setminus I)=4$. We further Assume that $H\setminus I$ has $\ell$ connected components $H_1,\ldots, H_{\ell}$.

		Let $H_j$ be a connected  component of $H\setminus I$ and assume that $H_j$ is isomorphic to the strong product of odd cycle  $C_{2t+1}$ and $K_2$ 
		 for some $t\geq 2$. 
		 Assume that $V(H_j)=\cup_{i=0}^{2t}\{x_i, x'_i\}$ and for each 
		 $0\leq i\leq 2t$ two vertices $x_i, x'_i$ are adjacent to each other and to $x_{i-1},x'_{i-1},x_{i+1},x'_{i+1}$ (with indices taken modulo $2t+1$).
Let $F_j\subset V(H_j)$ be such that $F_j$ is a maximum $K_4$-free subset of $V(H_j)$. 
Since $H_j$ is isomorphic to the strong product of $C_{2t+1}$ and $K_2$, it can be shown that $H_j\setminus F_j$ is a  graph with $t+1$ vertices and
exactly one edge. Furthermore, this edge this edge belongs to 
$\{x_ix_{i+1},x_ix'_{i+1},x'_ix_{i+1},x'_ix'_{i+1}\}$ for some $0\leq i\leq 2t$ (with indices taken modulo $2t+1$), as  illustrated in Figure~\ref{fi1}. 
For $u_i\in\{x_i,x'_i\}$ and  $u_{i+1}\in\{x_{i+1},x'_{i+1}\}$, let  $F_{_{u_iu_{i+1}}}$ denote
one of  a maximum $K_4$-free subset of $H_j$ such that $H\setminus F_{_{u_iu_{i+1}}}$ contains the edge $u_iu_{i+1}$.

 We shall prove that  there exists  a $F_j\subset V(H_j)$ such that $F_j$ is a maximum $K_4$-free subset of $H_j$
 such that  $H[I\cup (V(H_j)\setminus F_j)]$  contains no copy of $K_3$. Assuming that the stated assertion is false, 
consider $F_{_{x_0x_1}}$.
Therefore,   $H[I\cup\{x_0,x_1\}]$  contains a copy of $K_3$  
and there exists a vertex $u\in I$ such that $H[\{u, x_0,x_1\}]$ is isomorphic to $K_3$.
Since  $H_j$ is 5-regular and $\Delta(H)=6$, the vertex $u$ is  the only neighbor of  $x_0$ and $x_1$ in $I$.
Now consider $F_{_{x_0x'_1}}$,
the induced subgraph   $H[I\cup\{x_0,x'_1\}]$  contains a copy of $K_3$. Since  $u$ is  the only neighbor of  $x_0$ in $I$, 
we have $H[\{u, x_0,x'_1\}]$ is isomorphic to $K_3$, and consequently $u$ is adjacent to $x'_1$. 
By considering $F_{_{x_1x_2}}$ and  using a similar argument,  
  we can show that  $u$ must be adjacent to $x_2,x'_2$.
Again, using a similar argument for $F_{_{x_2x_3}}$,
   we can show that  $u$ must be adjacent to $x_3,x'_3$.
Therefore, the number of neighbors of $u$  is greater than $6$, which contradicts $\Delta(H)=6$.
So take a subset $F_j\subset V(H_j)$ such that $F_j$ is a maximum $K_4$-free subset of $H_j$
 and  $H[I\cup (V(H_j)\setminus F_j)]$  contains no copy of $K_3$.  Now, define $I'_j=V(H_j)\setminus F_j$.
 
%%%%%%%%%%%%%%%%%%%%%%%%%%%%%%%%%%%%%%%%%%%%%%%%%%%%%%%%%%%%%%%%%%%%%%%%%%%%%%%%%%%%%%%%%%%%%%%%
%%%%%%%%%%%%%%%%%%%%%%%%%%%%%%%%%%%%%%%%%%%%%%%%%%%%%%%%%%%%%%%%%%%%%%%%%%%%%%%%%%%%%%%%%%%%%%%%
%%%%%%%%%%%%%%%%%%%%%%%%%%%%%%%%%%%%%%%%%%%%%%%%%%%%%%%%%%%%%%%%%%%%%%%%%%%%%%%%%%%%%%%%%%%%%%%%
%%%%%%%%%%%%%%%%%%%%%%%%%%%%%%%%%%%%%%%%%%%%%%%%%%%%%%%%%%%%%%%%%%%%%%%%%%%%%%%%%%%%%%%%%%%%%%%%
%%%%%%%%%%%%%%%%%%%%%%%%%%%%%%%%%%%%%%%%%%%%%%%%%%%%%%%%%%%%%%%%%%%%%%%%%%%%%%%%%%%%%%%%%%%%%%%%
%%%%%%%%%%%%%%%%%%%%%%%%%%%%%%%%%%%%%%%%%%%%%%%%%%%%%%%%%%%%%%%%%%%%%%%%%%%%%%%%%%%%%%%%%%%%%%%%
%%%%%%%%%%%%%%%%%%%%%%%%%%%%%%%%%%%%%%%%%%%%%%%%%%%%%%%%%%%%%%%%%%%%%%%%%%%%%%%%%%%%%%%%%%%%%%%%
%%%%%%%%%%%%%%%%%%%%%%%%%%%%%%%%%%%%%%%%%%%%%%%%%%%%%%%%%%%%%%%%%%%%%%%%%%%%%%%%%%%%%%%%%%%%%%%%
%%%%%%%%%%%%%%%%%%%%%%%%%%%%%%%%%%%%%%%%%%%%%%%%%%%%%%%%%%%%%%%%%%%%%%%%%%%%%%%%%%%%%%%%%%%%%%%%
%%%%%%%%%%%%%%%%%%%%%%%%%%%%%%%%%%%%%%%%%%%%%%%%%%%%%%%%%%%%%%%%%%%%%%%%%%%%%%%%%%%%%%%%%%%%%%%%
%%%%%%%%%%%%%%%%%%%%%%%%%%%%%%%%%%%%%%%%%%%%%%%%%%%%%%%%%%%%%%%%%%%%%%%%%%%%%%%%%%%%%%%%%%%%%%%%
%%%%%%%%%%%%%%%%%%%%%%%%%%%%%%%%%%%%%%%%%%%%%%%%%%%%%%%%%%%%%%%%%%%%%%%%%%%%%%%%%%%%%%%%%%%%%%%%
 	\begin{figure}[ht]
		\begin{tabular}{ccc}

\begin{tikzpicture}[mystyle/.style={draw,shape=circle,fill=black}]
\node(x'0)[mystyle][label=below:$x'_0$] at (0.000000,1.450000) {};
\node(x'4)[mystyle][label=right:$x'_4$] at (-1.391268,0.370820) {};
\node(x'3)[mystyle][label=above:$x'_3$] at (-0.705342,-1.300820) {};
\node(x'2)[mystyle][label=above:$x'_2$] at (0.705342,-1.300820) {};
\node(x'1)[mystyle][label=left:$x'_1$] at (1.391268,0.370820) {};
\node(x0)[draw,shape=circle,fill=lightgray,label=above:$x_0$] at (0.000000,2.90000) {};
\node(x4)[mystyle][label=left:$x_4$] at (-2.782536,0.74164) {};
\node(x3)[mystyle][draw,shape=circle,fill=lightgray,label=left:$x_3$] at (-1.410684,-2.60164) {};
\node(x2)[mystyle][label=right:$x_2$] at (1.410684,-2.60164) {};
\node(x1)[draw,shape=circle,fill=lightgray,label=right:$x_1$] at (2.782536,0.74164) {};{{
\draw (x1) -- (x2);
    \draw (x2) -- (x3);
    \draw (x3) -- (x4);
    \draw (x4) -- (x0);
    \draw (x0) -- (x1);
    \draw (x'1) -- (x'2);
    \draw (x'2) -- (x'3);
    \draw (x'3) -- (x'4);
    \draw (x'4) -- (x'0);
    \draw (x'0) -- (x'1);    
    \draw (x1) -- (x'2);
    \draw (x2) -- (x'3);
    \draw (x3) -- (x'4);
    \draw (x4) -- (x'0);
    \draw (x0) -- (x'1);
    \draw (x'1) -- (x2);
    \draw (x'2) -- (x3);
    \draw (x'3) -- (x4);
    \draw (x'4) -- (x0);
    \draw (x'0) -- (x1);
    \draw (x1) -- (x'1);
    \draw (x2) -- (x'2);
    \draw (x3) -- (x'3);
    \draw (x4) -- (x'4);
    \draw (x0) -- (x'0);
}}
\end{tikzpicture}
		\end{tabular}\\
		\caption{The strong product of  $C_5$ and $K_2$. The induced subgraph on black vertices is  a maximum $K_4$-free subgraph, we call $F_{x_0x_1}$.}
		\label{fi1}
	\end{figure}
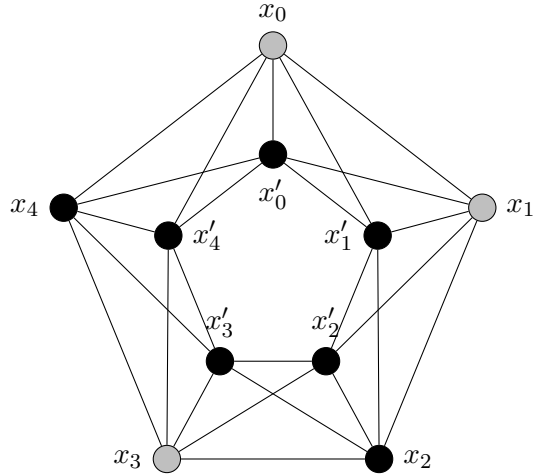
	
%%%%%%%%%%%%%%%%%%%%%%%%%%%%%%%%%%%%%%%%%%%%%%%%%%%%%%%%%%%%%%%%%%%%%%%%%%%%%%%%%%%%%%%%%%%%%%%%
%%%%%%%%%%%%%%%%%%%%%%%%%%%%%%%%%%%%%%%%%%%%%%%%%%%%%%%%%%%%%%%%%%%%%%%%%%%%%%%%%%%%%%%%%%%%%%%%
%%%%%%%%%%%%%%%%%%%%%%%%%%%%%%%%%%%%%%%%%%%%%%%%%%%%%%%%%%%%%%%%%%%%%%%%%%%%%%%%%%%%%%%%%%%%%%%%
%%%%%%%%%%%%%%%%%%%%%%%%%%%%%%%%%%%%%%%%%%%%%%%%%%%%%%%%%%%%%%%%%%%%%%%%%%%%%%%%%%%%%%%%%%%%%%%%
%%%%%%%%%%%%%%%%%%%%%%%%%%%%%%%%%%%%%%%%%%%%%%%%%%%%%%%%%%%%%%%%%%%%%%%%%%%%%%%%%%%%%%%%%%%%%%%%
%%%%%%%%%%%%%%%%%%%%%%%%%%%%%%%%%%%%%%%%%%%%%%%%%%%%%%%%%%%%%%%%%%%%%%%%%%%%%%%%%%%%%%%%%%%%%%%%
 
Now assume that $H_j$ is a connected component of $H\setminus I$ and
 it is not isomorphic to the strong product of $C_{2t+1}$ and $K_2$ for any $t\geq 2$,
		then Theorem~\ref{M.th2} implies that there is a maximal independent set
		 $I'_j\subseteq V(H_j)$ so that  $I'_j$ intersects all  cliques of size $4$ of $H_j$. 
		 It is clear that $H[I'_j]$ is $K_2$-free and $H_j\setminus I'_j$ is $K_4$-free
		
		Define $I'=\cup_{j=1}^{\ell} I'_j$,  $V_1=V(H)\setminus(I\cup I')$, and $V_2=I\cup I'$.
		 One can check that $H[V_1]$ is $K_4$-free and $H[V_2]$ is $K_3$-free.

		Suppose $H$ is a connected graph with maximum degree $\Delta(H)\geq 7$ and clique number $\omega(H)\leq \Delta(H)-1$.
		Assume that $p$ and $q$ are two positive integers where $p+q=\Delta(H)+1$ and $p\geq q\geq 2$. Assume that $\omega(H)=\Delta(H)-1$.
		If $H$ is isomorphic to odd cycle $C_{2t+1}$ and a complete graph $K_{r}$ for $t\geq 2$ and $r\geq 1$, then $\omega(H)=2r$ and $\Delta(H)=3r-1$.
		So $3r-2=2r$ and then $r=2$ and $\Delta(H)=5$, a contradiction.
		Therefore,  
		by using Theorem~\ref{M.th2}, we take a maximal independent set $I\subseteq V(H)$ 
		such that $I$ contains one vertex of any maximum clique of $H$. If
		$\omega(H)\leq \Delta(H)-2 $, only take a maximal independent set $I\subseteq V(H)$.
		In both cases, the graph $H\setminus I$ has maximum degree at most $\Delta(H)-1$ and clique number at most $\Delta(H)-2$.
		
Assume that $H\setminus I$ has $\ell$ connected components $H_1,\ldots, H_{\ell}$. Consider the connected component $H_i$ of $H\setminus I$.
 The graph $H_i$ has maximum degree at most $\Delta(H)-1$ and clique number at most $\Delta(H)-2$. 
 We may assume that $H_i$ has maximum degree  $\Delta(H)-1$.
 If $H_i$ does not have maximum degree $\Delta(H)-1$, let $v\in V(H_i)$ be a vertex with maximum degree in $H_i$ such that its degree is equal to 
			$\Delta(H_i)<\Delta(H)-1$. We add  $\Delta(H)-1-\Delta(H_i)$ new vertices to $H_i$ and join all of them to $v$, forming a new graph $H'_i$. The graph $H'_i$ 
			has maximum degree $\Delta(H)-1$ and clique number at most $\Delta(H)-2$.
  Thus, we can work with $H'_i$ instead of $H_i$.
   So, suppose  that $H_i$ has maximum degree  $\Delta(H)-1$ and clique number at most $\Delta(H)-2$. 

First assume that $q\geq 3$. Since $p+q=\Delta(H)+1$ and $\Delta(H)\geq 7$, we have  $p+(q-1)=(\Delta(H)-1)+1$ and $p+(q-1)\geq 7$. Thus, by using induction hypothesis, there exists  a partition of   $V(H_i)$ into $V_{i,1}, V_{i,2}$  such that $H[V_{i,1}]$ is $K_p$-free and  
		$H[V_{i,2}]$ is $K_{q-1}$-free. Define $V_1=\cup_{i=1}^{\ell}V_{i,1}$ and $V_2=(\cup_{i=1}^{\ell}V_{i,2})\cup I$.
				Since $I$ is independent set, we have $H[V_{2}]$ is $K_{q}$-free. Also, it is clear that  $H[V_{1}]$   is $K_{p}$-free.

Now assume that $q=2$.	Since $p+q=\Delta(H)+1$ and $\Delta(H)\geq 7$, we have  $(p-1)+q=(\Delta(H)-1)+1$ and $(p-1)+q\geq 7$. Thus, by using induction hypothesis, there exists 
a partition of   $V(H_i)$ into $V_{i,1}, V_{i,2}$  such that $H[V_{i,1}]$ is $K_{p-1}$-free and  
		$H[V_{i,2}]$ is $K_{q}$-free. Define $V_1=(\cup_{i=1}^{\ell}V_{i,1})\cup I$ and $V_2=\cup_{i=1}^{\ell} V_{i,2}$.
		Since $I$ is independent set, we have $H[V_{1}]$ is $K_{p}$-free. Also, it is clear that  $H[V_{2}]$   is $K_{q}$-free.

	\end{proof}
	
		In the sequel we will present the proof of Theorem~\ref{mth1}. Before we proceed with the detailed proof, let us provide an overview of the underlying idea.

	The main idea of the proof of Theorem~\ref{mth1} is that we  start with a partition of vertices of the graph $H$ into $S_0$ and $\overline{S_0}$ such that 
	$S_0$ is a maximum $K_p$-free subset of $H$, and subject to that $H[\overline{S_0}]$ has the minimum number of copies of $K_q$. This is somehow the best possible, 
	and if the number of copies of $K_q$ in $H[\overline{S_0}]$ is $0$, we have the required partition. We start with $S_0$ and then take a vertex $v_0$ in a copy of $K_q$ 
	in $H[\overline{S_0}]$ and move $v_0$ to $S_0$. Since $S_0$ is a maximum $K_p$-free subset, we obtain that $H[S_0\cup\{v_0\}]$ contains a copy of $K_p$. Then, we 
	try to find
	a vertex $y_0\neq v_0$ in $H[S_0\cup\{v_0\}]$ and move it from $S_0\cup\{v_0\}$ so that we can destroy all copies of $K_p$ in $H[S_0\cup\{v_0\}]$.
	We define ${S_1}\isdef ({S_0}\cup\{v_0\})\setminus\{y_0\}$.
	By carefully choosing $S_0$, $v_0$, and $v_0$, the vertex $y_0$ must be in a copy of $K_q$ in $H[\overline{S_1}]$.
	Then, we take a vertex $v_1$  in $H[\overline{S_1}]$, under certain assumptions, such that the number of copy of $K_q$ in $H[\overline{S_1}\setminus\{v_1\}]$
	is equal to that of $H[\overline{S_0}\setminus\{v_0\}]$.
		
		The key point is that  each vertex that moves into a part of the partition must lie a large clique in that part. The size of this clique is equal to the degree of 
	that vertex or one more.
	It is possible at the first step, we can find a copy of $K_{p+1}$, which leads us to contradiction. If not, we continue our procedure, and we will have 
	a sequence of $S_i$'s, $v_i$'s and $y_i$'s. 
		Since $H$ is finite, a carefully selected sequence of moves is guaranteed to eventually return to   itself in somehow, as we will elucidate in the 
		forthcoming proof. As a consequence, 
		 we will be able to find a copy of $K_{p+1}$ in $H$, which is a contradiction.

	The initial proof idea is influenced by the proof of Catlin and Lai's result on vertex arboricity~\cite{Catlin1}  as well as Catlin's  result on ordinary vertex 
	coloring~\cite{Catlin}.

In this regards, it is worth noting 
that Mozhan employed a related approach in ordinary coloring, utilizing a partition of the vertex set of graph into groups of color classes to establish bounds on the chromatic number in terms of 
 the maximum degree and clique number~\cite{mozhan}. This partition and and  its analogous partitions of graph vertices are commonly referred to as Mozhan's partitions.
The foundation of these concepts can be attributed to Lov{\'a}sz's seminal result in~\cite{Lovasz}, 
which states that given positive integers $d_1\geq d_2\geq\ldots\geq d_k$ such that $\sum_{i=1}^k d_i \geq \Delta(H)+1-k$, it is possible to partition the vertex set $V(H)$ into 
subsets $V_1,V_2,\ldots,V_k$ such that $\Delta(H[V_i])\leq d_i$ for all $1\leq i\leq k$. 
The proof is not hard: take a partition that minimizes the number of edges within the parts.
In~\cite{Catlin2}, Catlin enhanced the condition to $\sum_{i=1}^k d_i \geq \Delta(H)+2-k$ and expanded upon the idea by starting with a minimum partition and then moving 
vertices, while preserving minimality, until a required property is achieved.
 For more information about Mozhan's partition  and its usage see~\cite{mozhan, cranston2015,krab , rabern3} and for usage of Catlin's type partition 
 see~\cite{Bollob,Catlin2,rabern2}.

	\begin{proof}[\bf Proof of  Theorem \ref{mth1}]
		Without loss of generality suppose that $H$  is a connected graph. We have
		   $4\leq \omega(H)=p\leq \Delta(H)-2$. Since $\Delta(H)\geq 6$ and $q=\Delta(H)+1-p$, we have
		 		 $q\geq 3$. 
				 
	 Let $\F$ consist of all subsets $S\subseteq V(H)$ for which
 $H[S]$ is $K_p$-free and, subject to that, $S$ has the maximum possible size. 
If there is some $S\in \F$ for which  $H[\overline{S}]$ is  $K_{q}$-free, the statement holds.
Therefore,
 suppose that for any    $S\in \F$,  $H[\overline{S}]$ has a copy of $K_{q}$. 

In the sequel, we shall show  that a copy of $K_{p+1}$ is a subgraph of $H$, which  contradicts to $\omega(H)=p$.

		 As $S\in \F$, then by the maximality of $S$, for each vertex  $v\in  \overline{S}$, it 
		can be said that  $H[S\cup\{v\}]$ has a copy of $K_{p}$  that contains $v$. Therefore, $|N(v)\cap S|\geq {p-1}$ and $H[S]$ has a copy 
		of $K_{p-1}$. Now we have the following claim.
		\begin{claim}\label{f1} Let $S$ be a maximum $K_p$-free subset in $H$.
			Each vertex  $v$ of $\overline{S} $ either {\rm (a)}  lies in   at most two copies of $K_q$ in $H[\overline{S}]$ or 
			{\rm (b)} lies in a copy of $K_{q+1}$ that is a 
			connected component of $H[\overline{S}]$.	
		\end{claim}
		\begin{proof}[Proof of  Claim \ref{f1}]
		For each  vertex  $v$ of $\overline{S} $ we have $|N(v)\cap S|\geq {p-1}$. Since $q=\Delta(H)-(p-1)$, 
		 it follows that  the maximum degree of $H[\overline{S}]$ is at most $q$ and consequently  in $H[\overline{S}]$ each copy of $K_{q+1}$  is a connected component.					 
		 Suppose that (a) does not hold. Hence, there exists a vertex $v$ of $\overline{S}$ such that
			in $H[\overline{S}]$
			 vertex $v$ lies in at least three copies of $K_q$. Since $|N(v)\cap \overline{S}|\leq q$, the vertex $v$ must lie a copy of $K_{q+1}$ 
			 			in $H[\overline{S}]$.			
						This proves   Claim~\ref{f1}.

		\end{proof}

		Let  $S$ be a maximum $K_p$-free subset of $H$. For any vertex $v$ of ${\overline {S}}$, we have $K_p\subseteq H[S\cup\{v\}]$.
For the vertex $v$, define  $A_{v,{S}}$ as follows:
		\[ A_{v,S}\isdef\{y\in S\cup\{v\}~~|~~ y~{\rm lies ~in ~every ~copy ~of} ~K_p ~{\rm in }~H[S\cup \{v\}]\}.\]
			
		%========================================================================================
		%========================================================================================
		%========================================================================================
		%========================================================================================
	 		Take a member  $S_0\in \F$ such that $H[\overline{S_0}]$ has the minimum possible number of copies of $K_q$. 
Suppose that vertex $v_0$  in $\overline{S_0}$ is a fixed vertex such that $v_0$  belongs to a copy of $K_q$ in $H[\overline{S_0}]$.
		 By Claim~\ref{f1} the number of copies of   $K_q$ in $H[\overline{S_0}]$ that contain $v_0$ is equal to $1,2$ or $q$.

			 Since $v_0$ lies in a copy $K_q$ in $H[\overline{S_0}]$, we have $|N(v_0)\cap \overline{S_0}|\geq q-1$. Therefore, 
			  $|N(v_0)\cap S_0|\leq p$, otherwise,
		$\deg(v_0)\geq p+q=\Delta(H)+1$, a 
		contradiction.
		Since $H[S_0]$ is $K_p$-free and  $|N(v_0)\cap S_0|\leq p$, it follows that $v_0$ lies in at most two copies of $K_{p}$ in $H[S_0\cup\{v_0\}]$.
		 If the cardinality of the intersection of these two copies of $K_{p}$ is less than $p-1$, then 
		  $|N(v_0)\cap S_0|\geq p+1$, which again leads to a contradiction.
		 Therefore, $|A_{{v_0,S_0}}|\geq p-1$. 
		 Note that $v_0\in A_{v_0,S_0}$ 
		and since $p\geq 4$, we get that $|A_{v_0,S_0}|\geq 3$.
		From now on, if there is no ambiguity, we will write $A_{v_0}$ instead of $A_{v_0,S_0}$  for the  notational simplicity.
%========================================================================================
%========================================================================================
%========================================================================================
%========================================================================================
%========================================================================================
\begin{claim}\label{imp}
		Let $S$ be a maximum $K_p$-free subset of $H$. Let $v$ be a vertex in $\overline{S}$ such that $v$ belongs to a copy of 
		$K_q$ and the number of copies of $K_q$ in $H[\overline{S}\setminus\{v\}]$ is less than
		  the number of copies of $K_q$ in $H[\overline{S_0}]$. 
		Assume that  $y$  is an arbitrary vertex in $A_{v,S}$.
\begin{itemize}
		\item[(a)]  $(S\cup \{v\})\setminus \{y\}$ is a maximum $K_{p}$-free subset of $H$.
		\item[(b)]  The number of copies of $K_q$ in $H[(\overline{S}\setminus \{v\})\cup \{y\}]$  is greater than or equal to that of $H[\overline{S_0}]$.
                 \item[(c)]	The vertex $y$ belongs at least one  copy of $K_q$ in $H[(\overline{S}\setminus \{v\})\cup \{y\}]$.
		  \item[(d)] The vertex $y$ has at most $p$ neighbors in $S\cup\{v\}$.
		
		\end{itemize}
\end{claim}
	\begin{proof}[Proof of Claim~\ref{imp}]
	To prove (a), by the definition of $A_{v,S}$, we have  $(S\cup \{v\})\setminus \{y\}$ is $K_{p}$-free. 
		Since  $|(S\cup \{v\})\setminus \{y\}|=|S|$, it follows that $(S\cup \{v\})\setminus \{y\}$ is a $K_{p}$-free subset  of $H$ with maximum size. 
		
		To prove (b), since $\overline{S_0}$ is a maximum $K_p$-free subset of $H$ and subject to that 
		$H[\overline{S_0}]$ has the least possible copies of $K_q$, it follows that
		 the number of copies of $K_q$ in $H[(\overline{S}\setminus \{v\})\cup \{y\}]$ is not less than the number of copies  of 
		 $K_q$ in $H[\overline{S_0}]$. 
		 
		Part (c) is a direct consequence of (b).
		 To prove (d), by using (c), $y$ is in a copy of $K_q$ in $H[(\overline{S}\setminus \{v\})\cup \{y\}]$.
		 Therefore, $y$ has at least $q-1$ neighbors in $(\overline{S}\setminus \{v\})\cup \{y\}$ and consequently at most 
		 $p$ neighbors in $S\cup\{v\}$.
\end{proof}

		%========================================================================================
		%========================================================================================
		%========================================================================================
		%========================================================================================
		%========================================================================================
		%========================================================================================
		%========================================================================================
		%========================================================================================
		%========================================================================================
		%========================================================================================
		%========================================================================================
		%========================================================================================
		%========================================================================================
		 Choose a vertex $y_0\in A_{v_0}$. Define $S_1\isdef (S_0\cup\{v_0\})\setminus\{y_0\}$. 
		By Claim~\ref{imp}~(a), 
		 $S_1$ is a maximum $K_{p}$-free subset of $H$. Using  Claim~\ref{imp}~(c)  and the fact that
		 $v_0$ lies in at least one  copy of $K_q$ in $\overline{S_0}$, we have $y_0$ lies in a copy of $K_q$ in $H[\overline{S_1}]$.
		 Therefore, by Claim~\ref{f1}, $y_0$ exactly lies in $1,2,$ or $q$ copies of $K_q$ in $H[\overline{S_1}]$. 
		 
		 We define $B_{y_0}$ as follows. If $y_0$  lies in at most $2$ copies of $K_q$ in $H[\overline{S_1}]$, then define
\[B_{y_0}\isdef\{y\in \overline{S_1}~~|~~ y~{\rm lies ~in ~every ~copy ~of}~~~K_q ~{\rm in }~H[\overline{S_1}]~{\rm that~contains}~ y_0\}.\]
		 If $y_0$  lies in $q$ copies of $K_q$ in $H[\overline{S_1}]$, then define
\[B_{y_0}\isdef N[y_0]\cap {\overline {S_1}}.\]
		
		If $y_0$ lies one copy of $K_q$ in $H[\overline{S_1}]$, then $|B_{y_0}|=q$.
		If it lies $q$ copies of $K_q$ in $H[\overline{S_1}]$, by Claim~\ref{imp}~(b), we get that $y_0$ lies in a copy of $K_{q+1}$ which is a connected component 
		of $H[\overline{S_1}]$, then $B_{y_0}=N[y_0]\cap {\overline {S_1}}$ and consequently $|B_{y_0}|=q+1$.
		If $y_0$ lies two copies of $K_q$ in $H[\overline{S_1}]$, then the intersection of the vertices of  these two copies of $K_q$ is $B_{y_0}$.
				 If the size of the intersection of these two copies of $K_{q}$ is less than $q-1$, then 
		  $|N(y_0)\cap {\overline{S_1}}|\geq q+1$. On the other hand, since $y_0$ lies in $A_{v_0}$, it lies in a copy of $K_p$ in
		  $H[S_0\cup\{v_0\}]$. Therefore, $y_0$ has at least $p-1$ neighbors in $S_1=(S_0\cup\{v_0\})\setminus\{y_0\}$. Then,
		  $\deg(y_0)\geq p+q$,
		 which  leads to a contradiction.
		 Therefore $q-1\leq |B_{y_0}|\leq q+1$.
		 
		Since  $q\geq 3$, we have $|B_{y_0}\setminus{y_0}|\geq 1$. Thus, 
		there is at least one vertex $v_1$ in $B_{y_0}\setminus{y_0}$. By the definition of $B_{y_0}$, 
		we have $y_0$ and $v_1$ must be in the same number of copies of $K_q$ in $H[{\overline{S_1}}]$ and consequently
		 the number of copies of $K_q$ in 
		$H[{\overline {S_0}}\setminus\{v_0\}]$ is equal to that of $H[{\overline {S_1}}\setminus\{v_1\}]$. 
		
		By using  $S_1$ is a maximum $K_p$-free subset of $H$, we have $H[S_1\cup\{v_1\}]$ contains 
		at least one copy of $K_p$. Now consider  $A_{v_1, S_1}$, for simplicity we write $A_{v_1}$. 
		Similar to the argument for  $|A_{v_0}|$, we  obtain that $|A_{v_1}|\geq p-1 \geq 3$. 
		
		Now by considering  $A_{v_0}$ and  $A_{v_1}$, we have the following claim.
		
		%========================================================================================
		\begin{claim}\label{c6} If  $|A_{v_0}\cap A_{v_1}|\neq 0$, then $|A_{v_0}|=p$, $|A_{v_1}|=p$, and $K_{p+1}\subseteq H$.
		\end{claim}
		\begin{proof}[Proof of Claim~\ref{c6}] 
					 By way of contradiction, suppose  that  $|A_{v_0}|=p-1$.
			Consider a vertex $w\in A_{v_0}\cap A_{v_1}$. By Claim~\ref{imp}~(d), we know that both vertices $v_0$ and $w$
			 have at most $p$ neighbors in $S_0\cup\{v_0\}$.			
 From this fact and  $|A_{v_0}| = p-1$ we conclude that  $v_0$  lies in exactly two copies of $K_p$, 
denoted as $K_{p}$ and $K'_{p}$, in $H[S_0 \cup \{v_0\}]$. 
The intersection of $V(K_p)$ and $V(K'_p)$ is $A_{v_0}$. 
Let us assume that $V(K_p) \setminus A_{v_0} = \{x_{0}\}$ and $V(K'_p) \setminus A_{v_0} = \{x'_{0}\}$. 
As $S_0$ is $K_p$-free, the vertices $x_{0}$ and $x'_0$ are not adjacent.
Since  $w\in A_{v_0} $,  $w$ must be adjacent to all vertices of $A_{v_0}\setminus\{w\}$, $x_{0}$, and $x'_0$ in $H[S_0\cup \{v_0\}]$.

Considering that $w\in A_{v_1}$, the vertex  $w$ must be in  every copy of $K_p$ that contains $v_1$  in $H[S_1\cup\{v_1\}]$.
The vertex $w$ has $p$ neighbors in $S_1\cup\{v_1\}$, namely $v_1$, $x_{0}$, $x'_0$, and
all $p-3$ vertices of $A_{v_0}\setminus\{w, y_0\}$.
Since  $v_1$ and  $w$ are adjacent and both of them are in the same copies of $K_p$ in $H[S_1\cup\{v_1\}]$, 
$v_1$ must be adjacent to at least $p-2$ neighbors of $w$  in $H[S_1\cup\{v_1\}]$ that form a copy of $K_{p-2}$.
Note that
 since $x_{0}$ and $x'_0$ are not adjacent,  $v_1$ must be adjacent all $p-3$ vertices  of $A_{v_0}\setminus\{w, y_0\}$
 and at least one of $x_{0}$ and $x'_0$, say $x_0$.
Also $v_1$ is in $B_{y_0}$ and then is adjacent to $y_0$. Consequently, the induced subgraph on $A_{v_0}\cup\{x_0,v_1\}$ forms a copy $K_{p+1}$, a contradiction (see Figure~\ref{figa0}).
%%%%%%%%%%%%%%%%%%%%%%%%%%%%%%%%%%%%%%%%%%%%%%%%%%%%%%%%%%%%%%%%%%%%%%%%%%%%%%%%	
%%%%%%%%%%%%%%%%%%%%%%%%%%%%%%%%%%%%%%%%%%%%%%%%%%%%%%%%%%%%%%%%%%%%%%%%%%%%%%%%	
%%%%%%%%%%%%%%%%%%%%%%%%%%%%%%%%%%%%%%%%%%%%%%%%%%%%%%%%%%%%%%%%%%%%%%%%%%%%%%%%	
%%%%%%%%%%%%%%%%%%%%%%%%%%%%%%%%%%%%%%%%%%%%%%%%%%%%%%%%%%%%%%%%%%%%%%%%%%%%%%%%	
%%%%%%%%%%%%%%%%%%%%%%%%%%%%%%%%%%%%%%%%%%%%%%%%%%%%%%%%%%%%%%%%%%%%%%%%%%%%%%%%	
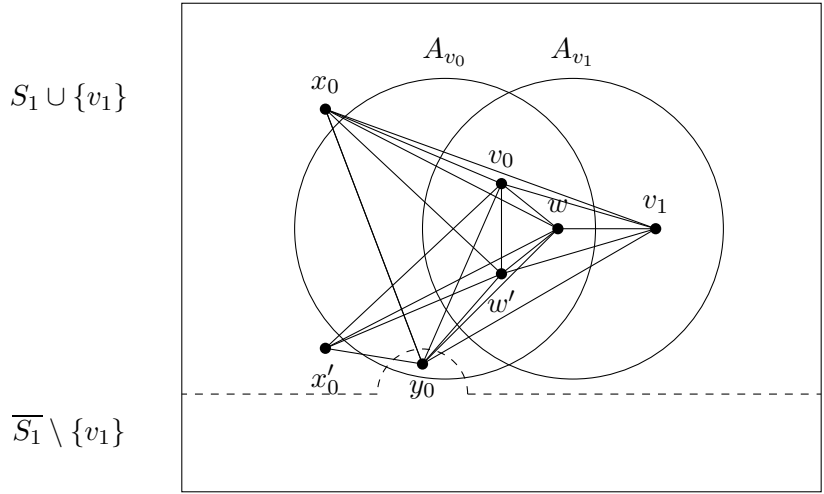
\begin{figure}[htbp]
    \centering
    \begin{tikzpicture}
        \draw (5,-3.5) rectangle (-3.5,3);

        % Circle 1
        \draw (0,0) circle (2cm);
        \node at (1.5,0) [fill=black, circle, inner sep=1.5pt, label=above:{$w$}] {};
        \node at (2.8,0) [fill=black, circle, inner sep=1.5pt, label=above:{$v_1$}] {};
        \node at (135:2.25cm) [fill=black, circle, inner sep=1.5pt, label=above:{$x_0$}] {};
        \node at (225:2.25cm) [fill=black, circle, inner sep=1.5pt, label=below:{$x'_0$}] {};
       \node at (-5,2.2) [ label=below:{$S_1\cup\{v_1\}$}] {};
        \node at (-5,-2.2) [ label=below:{${\overline{S_1}}\setminus\{v_1\}$}] {};

                % Circle 2
        \draw (1.7,0) circle (2cm);
        \node at (0.75,0.6) [fill=black, circle, inner sep=1.5pt, label=above:{$v_0$}] {};
        \node at (0.75,-0.6) [fill=black, circle, inner sep=1.5pt, label=below:{$w'$}] {};
        \node at (-0.3,-1.8) [fill=black, circle, inner sep=1.5pt, label=below:{$y_0$}] {};
        \node at (1.7,2.8) [ label=below:{$A_{v_1}$}] {};
        \node at (0,2.8) [ label=below:{$A_{v_0}$}] {};

        % Edges
        \draw (0.75,0.6) -- (1.5,0);
        \draw (0.75,-0.6) -- (1.5,0);
        \draw (0.75,0.6) -- (0.75,-0.6); 
         \draw (0.75,0.6) -- (-0.3,-1.8);
        \draw (0.75,-0.6) -- (-0.3,-1.8);
        \draw (-0.3,-1.8) -- (1.5,0); 
        \draw (0.75,0.6) -- (225:2.25cm) ;
        \draw ((225:2.25cm)  -- (-0.3,-1.8);
        \draw (225:2.25cm)  -- (1.5,0);
         \draw (0.75,0.6) -- (135:2.25cm) ;
         \draw (0.75,-0.6) -- (135:2.25cm) ;
         \draw (0.75,-0.6) -- (225:2.25cm) ;
        \draw ((135:2.25cm)  -- (-0.3,-1.8);
        \draw (135:2.25cm)  -- (1.5,0);
        \draw (135:2.25cm)  -- (-0.3,-1.8); 
        \draw ((2.8,0)  -- (-0.3,-1.8);
        \draw (2.8,0)  -- (1.5,0);
        \draw (2.8,0)  -- (0.75,0.6);
        \draw (2.8,0)  -- (0.75,-0.6);
        \draw (135:2.25cm)  -- (2.8,0);
                % Curve separating y_0
        \draw[dashed]  (0.3,-2.2) arc (0:180:0.6cm);
        \draw[dashed]  (0.3,-2.2) -- (5,-2.2);
        \draw[dashed]  (-0.9,-2.2) -- (-3.5,-2.2);

]
    \end{tikzpicture}
    \caption{Claim~{\ref{c6}}: $p=5$, $|A_{v_0}|=p-1=4$, we have a copy of $K_6$ with vertex set $\{v_0,y_0,v_1,w,w',x_0\}$.}
    		\label{figa0}
\end{figure}

%%%%%%%%%%%%%%%%%%%%%%%%%%%%%%%%%%%%%%%%%%%%%%%%%%%%%%%%%%%%%%%%%%%%%%%%%%%%%%%%	
%%%%%%%%%%%%%%%%%%%%%%%%%%%%%%%%%%%%%%%%%%%%%%%%%%%%%%%%%%%%%%%%%%%%%%%%%%%%%%%%	
%%%%%%%%%%%%%%%%%%%%%%%%%%%%%%%%%%%%%%%%%%%%%%%%%%%%%%%%%%%%%%%%%%%%%%%%%%%%%%%%	
%%%%%%%%%%%%%%%%%%%%%%%%%%%%%%%%%%%%%%%%%%%%%%%%%%%%%%%%%%%%%%%%%%%%%%%%%%%%%%%%	
%%%%%%%%%%%%%%%%%%%%%%%%%%%%%%%%%%%%%%%%%%%%%%%%%%%%%%%%%%%%%%%%%%%%%%%%%%%%%%%%	
			Now suppose, by way of contradiction, that  $|A_{v_1}|=p-1$.
			%Let $w$ be in $A_{v_0}\cap A_{v_1}$.
			Therefore,  $v_1$ lies in the two copies of $K_p$ in $H[S_1\cup \{v_1\}]$.
			Since $S_1$ is $K_p$-free  and $H[S_1\cup\{v_1\}]$ has two copies of $K_p$, there are two nonadjacent vertices $x_1, x'_1$ in $S_1$ such that
                          both of them are adjacent to all vertices of $A_{v_1}$, including $w$.
                          
                          If $v_1$ is adjacent all vertices of $A_{v_0}$, then $H[A_{v_0}\cup\{v_1\}]$ is isomorphic to $K_{p+1}$, a contradiction. 
                          Assume that there is a vertex $\hat v$ in $A_{v_0}$  that is not adjacent to $v_1$.

                           We shall show that $w$ has $p+1$ neighbors in $S_0\cup\{v_0\}$, which is contradiction by Claim~\ref{imp}~(d).
                          The vertex $w$ is adjacent to $x_1, x'_1, {\hat v}, y_0$, and all $p-3$ vertices  in $A_{v_1}\setminus\{w,v_1\}$. 
                          	Therefore, we can suppose that $|A_{v_1}|=|A_{v_0}|=p$.

			Note that $y_0$ and $v_0$ belong to $A_{v_0}$. 
						 The vertex $y_0$ is not in $S_1$, this yields $y_0\not\in A_{v_1}$. Therefore, $|A_{v_0}\cap A_{v_1}|\leq p-1$.

			Assume that  $|A_{v_0}\cap A_{v_1}|\leq p-3$. For a vertex  $w\in A_{v_0}\cap A_{v_1}$,  it has $p-1$ neighbors in $A_{v_0}$. We shall show
			the vertex $w$ has at least two other neighbors in $S_0\setminus A_{v_0}$. 	Note that $|A_{v_1}\setminus (A_{v_0}\cup\{v_1\})|\geq 2$.
			So, 
			we can take two vertices in $A_{v_1}\setminus (A_{v_0}\cup\{v_1\})$  as neighbors of $w$ in $S_0\setminus A_{v_0}$.
		Then , we have $|N(w)\cap (S_0\cup\{v_0\})|\geq p+1$, which is contradiction with Claim~\ref{imp}~(d).
			Hence, we can assume that $|A_{v_0}\cap A_{v_1}|\geq p-2$.

			If  $|A_{v_0}\cap A_{v_1}|= p-1$. Therefore, $A_{v_0}\setminus\{y_0\} \subset A_{v_1}$. Therefore, $v_1$ is adjacent to
			all vertices of $A_{v_0}$. So, 
			we have a copy of $K_{p+1}$ with vertex set  $A_{v_0}\cup\{v_1\}$.
			
			 Now assume that $|A_{v_0}\cap A_{v_1}|= p-2$.
			 Let $A_{v_0}\setminus A_{v_1}=\{y_0, {\hat v}\}$.
			Note that $v_1$ is not adjacent to ${\hat v}$. If $v_1$ is  adjacent to ${\hat v}$, then
			 $H[A_{v_0}\cup \{v_1\}]$ is isomorphic to $K_{p+1}$.
			Each vertex  $w \in A_{v_0}\cap A_{v_1}$ is adjacent to $p-1$ neighbors in $A_{v_1}$ and another vertex in $A_{v_0}\setminus A_{v_1}$.
			By using Claim~\ref{imp}~(c), the vertex $y_0$  lies in at least a copy of $K_q$ in   $H[\overline{S_1}]$. Then,
			 $y_0$ has at least $q-1$ neighbors in $\overline{S_1}$ and consequently
			  $y_0$ has at least $q-2$ neighbors in $\overline{S_1}\setminus \{v_1\}$. On the other hand, 
                           since $y_0$  lies in $A_{v_0}$ and a copy of $K_p$ in $H[S_0\cup\{v_0\}]$, it yields  $y_0$ has at least $p-1$
                           neighbors in $S_0\cup\{v_0\}$	 and consequently  
$y_0$ has at least $p$
                           neighbors in $S_1\cup\{v_1\}$. Therefore, $y_0$ has at most $q-1$ neighbors in $\overline{S_1}\setminus \{v_1\}$. Hence,
                           			  we can consider the the following two cases.
\item[\bf Case 1:] $y_0$ has exactly $q-1$ neighbors in $\overline{S_1}\setminus \{v_1\}$.
			Therefore, $y_0$ has at most $p$ neighbors in $S_1\cup\{v_1\}$.
			Assume that $w\in A_{v_0}\cap A_{v_1}$. As $y_0$ has $p-1$ neighbors in $(S_1\cup\{v_1\})\setminus\{w\}$ and 
			two of them, $v_1$ and $\hat{v}$ are not adjacent, it yields  $H[(S_1\cup\{v_1,y_0\})\setminus\{w\}]$ does not contain any copy of $K_p$,
			which contradicts with $S_0$ is a maximum $K_p$-free subset of $H$. 
			
\item[\bf Case 2:] 

$y_0$ has exactly $q-2$ neighbors in $\overline{S_1}\setminus \{v_1\}$. 
			Each $w\in A_{v_0}\cap A_{v_1}$ has at least $p$ neighbors in
			$S_1\cup \{v_1\}$ and consequently at most $q-1$ neighbors in
			$\overline{S_1}\setminus \{v_1\}$.
			Since $H[{\overline {S_0}}]$ has  the minimum number of copies of $K_q$  and 
			$v_0$ lies in at least one copy of $K_q$ in $H[{\overline {S_0}}]$, 
			it implies $y_0$ must lie at least one copy of $K_q$ in $H[{\overline {S_1}}]$.   We took $v_1\in B_{y_0}$. 
			Therefore,   $w$ must be contained in  at least one copy of $K_q$,
			say $K_q^w$, in $H[({\overline {S_1}}\setminus\{v_1\})\cup\{w\}]$.
			 As $y_0w\in E(H)$ and
			$|N(w)\cap (\overline{S_1} \setminus\{v_1\})|\leq q-1$, we  deduce that 
			   $y_0$ must lie in  $K_q^w$ in $H[(\overline{S_1} \setminus\{v_1\})\cup\{w\}]$. 
			Since $y_0$ has exactly $q-2$ in $\overline{S_1}\setminus \{v_1\}$, it yields that
			$K_q^w$ consists of $w,y_0$, and $q-2$ neighbors of $y_0$  in $\overline{S_1}\setminus \{v_1\}$.
			Therefore each vertex of $A_{v_0}\cap A_{v_1}$ is adjacent to $y_0$, and its $q-2$ neighbors  in $\overline{S_1}\setminus \{v_1\}$.
			Also, we have  $v_1$ is adjacent to $y_0$, and its $q-2$ neighbors  in $\overline{S_1}\setminus \{v_1\}$.
Therefore, the induced subgraph on $A_{v_0}\cap A_{v_1}$, $v_1,y_0$, and $q-2$ neighbors of $y_0$ in $\overline{S_1}\setminus \{v_1\}$ 
is isomorphic to 	$K_{p+q-2}$,
a contradiction.

		\end{proof}
		%========================================================================================
		
		Using  Claim \ref{c6}, we can assume   that $A_{v_0}\cap A_{v_1}= \varnothing$. Let $y_1$ be an arbitrary vertex of $A_{v_1}\setminus\{v_1\}$.
		 For $i\geq 1$,  suppose that we have already defined $S_{i}, B_{y_{i-1}}, v_{i}, A_{v_{i}}$, and $y_{i}$.
		If we have $A_{v_j}\cap A_{v_i}= \varnothing$ for each $0\leq j<i$, we proceed to  define $S_{i+1}$, $B_{y_{i}}$, $v_{i+1}$, $A_{v_{i+1}}$, and $y_{i+1}$ in the following manner:
		we first define $S_{i+1}$, which allows us to define $B_{y_{i}}$. Then, we define $v_{i+1}$,  $A_{v_{i+1}}$. Finally, we define $y_{i+1}$.
		
\begin{enumerate} 
\item $S_{i+1}\isdef(S_{i}\cup\{v_{i}\})\setminus\{y_{i}\}$.	
\item If $y_{i}$ lies in at most $2$ copies of $K_q$ in $H[\overline{S_{i+1}}]$, define 
$$B_{y_{i}}\isdef\{y\in \overline{S_{i+1}}~~|~~ y~{\rm lies ~in ~every ~copy ~of}~~~~~~~~~K_q ~{\rm in }~H[\overline{S_{i+1}}]~{\rm that~ contains}~y_{i}\}.$$
If $y_{i}$ lies in $q$ copies of $K_q$ in $H[\overline{S_{i+1}}]$, define
 \[B_{y_{i}}\isdef N[y_{i}]\cap \overline{S_{i+1}}.\]

\item Take $v_{i+1}\in B_{y_{i}}\setminus\{y_0,\ldots,y_{i}\}$.
\item $A_{v_{i+1}}\isdef A_{v_{i+1}, S_{i+1}}$.
\item Take $y_{i+1}\in A_{v_{i+1}}\setminus\{v_{i+1}\}$.

\end{enumerate} 

Note that $S_0$ and $S_1$ are  $K_p$-free subsets of $H$ with maximum size. Assume that  $S_i$ is a $K_p$-free subset of $H$.
   By applying Claim~\ref{imp}~(a), we get that  $H[S_{i+1}]$ is $K_p$-free, too. Since  $|S_{i+1}|=|S_{i}|=|S_{0}|$, we get that
   $S_{i+1}$ is a $K_p$-free subset of $H$ with maximum size.
				Thus, the induced subgraph $H[S_{i+1}\cup\{v_{i+1}\}]$ has a copy of $K_p$.
				Since $H[S_{i+1}]$ is $K_p$-free and $|N(v_{i+1})\cap S_{i+1}|\leq p$,
				we can use a similar argument as  $|A_{v_0}|$ and  obtain that $|A_{v_{i+1}}|\geq p-1 \geq 3$. 
		We  note that $v_{i+1}\in A_{v_{i+1}}$. 

 Since $q\geq 3$, Parts (c) and (d) of the next  claim allow us  to choose a vertex $v_{i+1}$ in $B_{y_{i}}\setminus\{y_0,\ldots,y_{i}\}$.

		\begin {claim}\label{0p1}
		Let $i$ be a positive integer such that for each pair $j<j'\leq i$, we have $A_{v_j}\cap A_{v_{j'}}=\varnothing$.
		 Let $\iota\leq i$  be a nonnegative integer. Then, the following items hold. 
		\begin{itemize}
\item[(a)]  $A_{v_\iota}\setminus\{y_\iota\}\subseteq S_{i+1}$. 
\item[(b)] If    $|A_{v_{\iota}}|=p-1$, 
		  then there exist two nonadjacent vertices $x_{\iota}$ and $x'_{\iota}$ in  $N(v_{\iota})\cap    
		  S_{\iota}$  that  are not in $A_{v_{\iota}}$, and are adjacent to all vertices in $A_{v_{\iota}}$. 
		  Furthermore both vertices $x_{\iota}$ and $x'_{\iota}$ belong to  $S_{i+1}$ and do not belong $A_{v_{i}}$.
\item[(c)] If  $|B_{y_{i}}|=q-1$, 
we have $|B_{y_{i}}\cap \{y_0,y_1, \ldots, y_{i-1}\}|= 0$.
\item[(d)] If $|B_{y_{i}}|\in\{q,q+1\}$, we have $|B_{y_{i}}\cap \{y_0,y_1, \ldots, y_{i-1}\}|\leq 1$.
\item[(e)] The vertex  $y_\iota$ belongs  to ${\overline {S_{i+1}}}$.
\end{itemize}

\end{claim}
\begin{proof}[Proof of Claim~\ref{0p1}] 
To prove (a),
 by using the definition of $S_{\iota+1}$,  we have $A_{v_\iota}\setminus\{y_\iota\}\subseteq S_{\iota+1}$.
 For $\iota+1\leq i'\leq i$,
 assume that  and  $A_{v_\iota}\setminus\{y_\iota\}\subseteq S_{i'}$.
		Since $(A_{v_{\iota}}\setminus\{y_\iota\})\cap A_{v_{i'}}=\varnothing$
		and  we choose  $y_{i'}$ from $A_{v_{i'}}$, it follows that
		 $A_{v_\iota}\setminus\{y_\iota\}\subseteq (S_{i'}\cup\{v_{i'}\})\setminus\{y_{i'}\}= S_{i'+1}$.

To prove (b), note that $v_{\iota}$ lies at least one copy of $K_q$ in $\overline{S_{\iota}}$ and 
consequently it has at most $p$ neighbors in $S_{\iota}$.
 From this fact and  $|A_{v_\iota}| = p-1$ we obtain  $v_\iota$ lies in exactly two copies of $K_p$, 
denoted as $K_{p}$ and $K'_{p}$, in $H[S_\iota \cup \{v_\iota\}]$. The intersection of $V(K_p)$ and $V(K'_p)$ is $A_{v_\iota}$. 
Let us assume that $V(K_p) \setminus A_{v_\iota} = \{x_{\iota}\}$ and $V(K'_p) \setminus A_{v_\iota} = \{x'_{\iota}\}$. 
Then we have $x_{\iota}$ and $x'_\iota$ are not in $A_{v_\iota}$.
Since $S_\iota$ is $K_p$-free, the vertices $x_{\iota}$ and $x'_\iota$ are not adjacent.

\noindent Now  assume that $\iota<i$ and  $x_\iota$ is  in $A_{v_{i}}$, then $x_\iota$ has at least $p-2$ neighbors in $A_{v_{i}}$. 
 We know that $|A_{v_\iota}\setminus\{y_\iota\}|=p-2$. Since $p\geq 4$, 
 the vertex $x_\iota$ has at least two neighbors in $A_{v_\iota}\setminus\{y_\iota\}$.
 By~(a), $A_{v_\iota}\setminus\{y_\iota\}\subseteq S_{i}$. Therefore, 
 since these  two neighbors of $x_\iota$  are adjacent, at least one of them must be in
$A_{v_{i}}$. 
This implies that $A_{v_\iota} \cap A_{v_{i}} \neq \varnothing$, which leads to a contradiction. Therefore, $x_\iota$ is not in $A_{v_{i}}$. 
To prove $x_\iota$ is  in $S_{i+1}$. Note that $x_\iota$ is  in $S_{\iota+1}$. For $i'$ where $\iota+1\leq i'\leq i$, assume that $x_\iota\in S_{i'}$. 
Since $x_\iota\not \in A_{v_{i'}}$, we have $x_\iota\neq y_{i'}$. As $S_{i'+1}\isdef (S_{i'}\cup\{v_{i'}\})\setminus\{y_{i'}\}$, we conclude that
$x_\iota$ is in $S_{i'+1}$. By symmetry, we can conclude that $x'_\iota$ is  in $S_{i+1}$ and is not in $A_{v_{i}}$.

To prove (c),
by  way of contradiction, suppose that for some  $j< i$ we have $y_j\in B_{y_{i}} $.
			Since $|B_{y_{i}}|=q-1$,  $y_j$ lies in two copies of $K_q$ in $H[\overline{S_{i+1}}]$. 
			Therefore, the vertex $y_j$ must have $q$ neighbors in ${\overline {S_{i+1}}}$. Also, $y_j$ is adjacent to $v_{j+1}$.
			 Since $j< i$,  from (a)  we have $(A_{v_j}\setminus\{y_j\})\subseteq S_{i+1}$ 
			 and $(A_{v_{j+1}}\setminus\{y_{j+1}\})\subseteq S_{i+1}$. 
			 The last inclusion relation yields $v_{j+1}\in S_{i+1}$.

			If $|A_{v_j}|=p$, then $y_j$ has $p-1$ neighbors in $A_{v_{j}}\setminus\{y_j\}$. Also, it is adjacent $v_{j+1}$ and $q$ vertices in 
			$\overline{S_{i+1}}$.
			Therefore, $\deg(y_j)\geq p+q$ which is not possible. 
			
			Assume that $|A_{v_j}|=p-1$. Then, $|N(y_j)\cap A_{v_j}|=p-2$ 
			and from (b) there are two nonadjacent vertices   $x_j,x'_j$ in $N(y_j)\cap S_j$ 
			which are not in $A_{v_j}$ and $\{x_j,x'_j\}\subseteq S_{i+1}$.
			Then, in $H[S_{i+1}]$, the vertex $y_j$ is adjacent to
$v_{j+1}$, all $p-2$ vertices in $A_{v_j}\setminus\{y_j\}$, $x_j$, and $x'_j$.
Thus, $y_j$ has $p+1$ neighbors in
			$S_{i+1}$.  Since $y_j$ has $q$ neighbors in ${\overline {S_{i+1}}}$, it follows that  $\deg(y_j)\geq p+q+1$, which is not possible. 

To prove (d),  by way of contradiction, suppose that $|B_{y_{i}}\cap \{y_0,y_1, \ldots, y_{i-1}\}|\geq 2$. 
 Assume that $\{y_{\iota}, y_{\iota'}\}\subset B_{y_{i}}$ where $\iota<\iota'\leq i-1$. 
			As  $|A_{v_\iota}\cap A_{v_{\iota'}}|= 0$ and $|A_{v_\iota}\cap A_{v_{i}}|= 0$,
			we have $y_{\iota'}\notin A_{v_{\iota}}$ and $y_{i}\notin A_{v_{\iota}}$. 
			
			The size of $A_{v_\iota}$ is  $p$ or $p-1$. 
			First assume that   $|A_{v_\iota}|=p$. Therefore, $y_{\iota}$  is adjacent to $y_{\iota'}, y_{i} $, and
			$p-1$ vertices of $A_{v_\iota}\setminus\{y_{\iota}\}$ in $S_{\iota}\cup\{v_{\iota}\}$. Thus, in this case
			we have $|N(y_{\iota})\cap (S_{\iota}\cup\{v_{\iota}\})|\geq p+1$, a contradiction 
			with Claim~\ref{imp}~(d).

			Assume that $|A_{v_\iota}|=p-1$. There exist two nonadjacent vertices $x_{\iota}$ and $x'_{\iota}$ in  $N(v_{\iota})\cap S_{\iota}$
 that  are not in $A_{v_\iota}$, and are adjacent to all vertices in $A_{v_\iota}$.
 Since $y_i$ and $y_{\iota'}$ are adjacent, we have $\{y_i,y_{\iota'}\}\neq \{x_{\iota},x'_{\iota}\}$.
 Therefore, $y_{\iota}$  is adjacent to $y_{\iota'}, y_{i}, x_{\iota},x'_{\iota}$, and
			$p-2$ vertices of $A_{v_\iota}\setminus\{y_{\iota}\}$ in $S_{\iota}\cup\{v_{\iota}\}$. Therefore, in this case
			 $|N(y_{\iota})\cap (S_{\iota}\cup\{v_{\iota}\})|\geq p+1$, a contradiction 
			with Claim~\ref{imp}~(d).
			
To prove (e), note that $y_\iota$ is in $(\overline{S_{\iota}}\setminus\{v_{\iota}\})\cup \{y_{\iota}\}=\overline{S_{\iota+1}}$. 
For $\iota+1\leq i$, assume that  $y_\iota\in{\overline {S_{i}}}$.
Since we choose $v_{i}$ in $B_{y_{i-1}}\setminus\{y_0,y_1,\ldots, y_{i-1}\}\subset \overline{S_{i}}$,  it implies that $y_\iota\neq v_{i}$. 
Therefore,  $y_\iota$ belongs to
$ (\overline{S_{i}}\setminus\{v_{i}\})\cup\{y_{i}\}= \overline{S_{i+1}}$.

\end{proof}

%========================================================================================
%========================================================================================
%========================================================================================
%========================================================================================
%========================================================================================
%========================================================================================
%========================================================================================

		%%%%%%%%%%%%====================================================
		%%%%%%%%%%%%====================================================
		%%%%%%%%%%%%====================================================
		%%%%%%%%%%%%====================================================
		%%%%%%%%%%%%====================================================
		%%%%%%%%%%%%====================================================
		Now assume that we defined $A_{v_{0}},A_{v_{1}},\ldots, A_{v_{\ell}}$ such that $\ell$ is the maximum number for which 
		$A_{v_j}\cap A_{v_{j'}}= \varnothing$ for each pair $j$ and $j'$ where $j<j'\leq \ell-1$. Since $H$ is a finite graph, such an $\ell$ exists.
		Hence, there exists some $i_0\leq \ell-1$ such that $|A_{v_{i_0}}\cap A_{v_{\ell}}|\neq 0$.
		Since we choose $v_{\ell}\in B_{y_{\ell-1}}\setminus \{y_0,y_1, \ldots, y_{\ell-1}\}$, it implies that
		$v_{\ell}\neq y_{i_0}$ and consequently $A_{v_{i_0}}\neq A_{v_{\ell}}$. It follows that $|A_{v_{i_0}}\cap A_{v_{\ell}}|\leq p-1$.
		  
	\noindent	Without loss of generality, we can  assume that $i_0=0$.  	
			By using this fact that $1\leq | A_{v_0}\cap A_{v_{\ell}} |\leq p-1$, we shall show that a copy $K_{p+1}$ is contained in $H$.

%========================================================================================
%========================================================================================
%========================================================================================
%========================================================================================
%========================================================================================
%========================================================================================

\begin{claim}\label{intp2}
$| A_{v_0}\cap A_{v_{\ell}}|\geq p-2$.
\end{claim}		
\begin{proof}[Proof of Claim~\ref{intp2}] 
Assume that $| A_{v_0}\cap A_{v_{\ell}}|\leq p-3$. Let $w\in  A_{v_0}\cap A_{v_{\ell}}$. In the following we show that regardless of the sizes of $A_{v_0}$ and $A_{v_{\ell}}$, 
vertex $w$ has at least $p+1$ neighbors in $S_\ell\cup\{v_{\ell}\}$, which contradicts Claim~\ref{imp}~(d).

If $|A_{v_0}|=|A_{v_\ell}|=p$, then $w$ has $p-1$ neighbors in $A_{v_\ell}$ and $2$ neighbors in $A_{v_0}\setminus(\{y_0\}\cup A_{v_\ell})$. Thus, $w$ has at least $p+1$ neighbors in $S_\ell\cup\{v_\ell\}$.

If $|A_{v_{0}}|=p-1$ and $|A_{v_{\ell}}|=p$, then vertex $w$ has $p-1$ neighbors in $A_{v_{\ell}}$. Since $|A_{v_{0}}|=p-1$, then there is at least one vertex $\hat{y}$ in $A_{v_0}\setminus( A_{v_{\ell}}\cup\{y_0\})$ and by Claim~\ref{0p1}~(b),
 there are two nonadjacent vertices $x_{0}$ and $x'_{0}$ in $N(v_{0})\cap S_{0}$ which are not in $A_{v_0}$ and are in $S_{\ell}$.
 Since $H[A_{v_{\ell}}]$ is isomorphic to $K_p$, at most one of $x_0$ and $x'_0$ cannot be in $A_{v_{\ell}}$. 
Therefore, $w$ has at least $p+1$ neighbors in $A_{v_\ell}\cup\{ x_{0}, x'_{0}, \hat{y}\}$.

Assume that $|A_{v_{\ell}}|=p-1$. Therefore, there are two nonadjacent vertices $x_{\ell}$ and $x'_{\ell}$ in $N(v_{\ell})\cap S_{\ell}$ which are not in $A_{v_\ell}$ and are adjacent to all vertices in $A_{v_\ell}$. Hence, $w$ has $p$ neighbors in $A_{v_\ell}\cup\{x_{\ell}, x'_{\ell}\}$. If $|A_{v_{0}}|=p$, then there are two vertices in $A_{v_0}\setminus(\{y_0\}\cup A_{v_{\ell}})$. Since $x_{\ell}$ and $x'_{\ell}$ are not adjacent and $H[A_{v_{0}}]$ is isomorphic to $K_p$, we conclude that there is at least one vertex in $A_{v_0}\setminus(A_{v_\ell}\cup\{x_{\ell}, x'_{\ell}, y_0\})$. Therefore, $w$ has at least $p+1$ neighbors in $S_{\ell}\cup\{v_\ell\}$.
 
If $|A_{v_{0}}|=p-1$, then there is at least one vertex $\hat{y}$ in $A_{v_0}\setminus( A_{v_{\ell}}\cup\{y_0\})$ and
by Claim~\ref{0p1}~(b) there are two nonadjacent vertices $x_{0}$ and $x'_{0}$ in $N(v_{0})\cap S_{0}$ which are not in $A_{v_0}$ and 
 are in $S_{\ell}$. Therefore,  one can check that $w$ has at least $p+1$ neighbors in $A_{v_\ell}\cup\{x_{\ell}, x'_{\ell}\}\cup\{ x_{0}, x'_{0}, \hat{y}\}$.
\end{proof}

\begin{claim}
If $| A_{v_0}\cap A_{v_{\ell}}|\geq p-2$, then
$| A_{v_0}|=p$ and  $|A_{v_{\ell}}|= p$.
\end{claim}	
%========================================================================================
%========================================================================================
%========================================================================================
%========================================================================================
		
\begin{proof}

\noindent Assume that $|A_{v_0}|=p-1$. Since $| A_{v_0}\cap A_{v_{\ell}}|\geq p-2$, and $y_0\not \in A_{v_\ell}$ we have $| A_{v_0}\cap A_{v_{\ell}}|= p-2$.
We have $H[S_{\ell}\cup\{v_\ell\}]$ contains a copy of $K_p$. By the definition of $A_{v_{\ell}}$, this copy consists of $v_{\ell}$, $ A_{v_0}\cap A_{v_{\ell}}$ and 
another vertex $z$. The vertex $z$ is one of $x_0$ and $x'_0$; otherwise each vertex $w\in A_{v_0}\cap A_{v_{\ell}}$ has at least $p+1$ neighbors in $S_{\ell}\cup\{v_\ell\}$,
which contradicts with Claim~\ref{imp}~(d). By symmetry assume that $z=x_0$. Therefore, in $S_{\ell}$ the vertex
$v_{\ell}$ is adjacent to $p-2$ vertices of $ A_{v_0}\cap A_{v_{\ell}}$ and $x_0$.

If $v_{\ell}$ is adjacent to $y_0$, then $H[A_{v_0}\cup\{v_\ell , x_0\}]$ is isomorphic to $K_{p+1}$, a contradiction.
So we may assume that $v_{\ell}$ is not adjacent to $y_0$.

Since for each $w\in A_{v_\ell}\cap A_{v_0}$  we have $({S_\ell}\cup\{v_{\ell}\})\setminus\{w\}$ is maximum $K_p$-free and by Claim~\ref{imp}~(c),
$w$ is  in a copy of $K_q$ in  $H[(\overline{S_\ell}\setminus\{v_{\ell}\})\cup\{w\}]$. Since $w$ has $p$ neighbors in ${S_\ell}\cup\{v_{\ell}\}$, it implies that
$w$ has $q-1$ neighbors in $\overline{S_\ell}\setminus\{v_{\ell}\}$.

The vertex $w$ is  in a copy of $K_q$ in  $H[(\overline{S_\ell}\setminus\{v_{\ell}\})\cup\{w\}]$,
so $w$ has exactly xt $q-1$ neighbors in $\overline{S_\ell}\setminus\{v_{\ell}\}$, and
 $w$ is adjacent to $y_0$. Therefore, $y_0$ must be in this copy $K_q$ in  $H[(\overline{S_\ell}\setminus\{v_{\ell}\})\cup\{w\}]$ and consequently 
$y_0$ has  $q-2$ neighbors in $\overline{S_\ell}$. Hence,
 $y_0$ has $p+1$ neighbors in $S_\ell$, namely  $p-2$ vertices in $A_{v_0}\cap A_{v_\ell}$, $x_0$, $x'_0$, and $v_1$ in $S_\ell$. 
Each vertex $w$ in $A_{v_0}\cap A_{v_\ell}$ is not adjacent to $v_1$; 
otherwise if $v_1$ is adjacent to  $w_0\in A_{v_0}\cap A_{v_\ell}$, then $w_0$ has $p+1$ neighbors $S_\ell$, a contradiction with Claim~\ref{imp}~(d).
 So
each copy of $K_p$ in $H[S_{\ell}\cup\{v_\ell, y_0\}]$ must contain at least one vertex ${\hat w}$ in $A_{v_0}\cap A_{v_\ell}$. 
Therefore, the induced subgraph of $H$ on $(S_{\ell}\cup\{v_\ell, y_0\})\setminus\{{\hat w}\}$ is $K_p$-free, a contradiction with $S_{\ell}$
 is maximum  $K_p$-free subset of $H$. Therefore we can assume that $|A_{v_0}|=p$.

 Assume that $|A_{v_\ell}|=p-1$. Since $|A_{v_0}\cap A_{v_\ell}|\geq p-2$ and  $v_{\ell}\not\in A_{v_0}$, we have $|A_{v_0}\cap A_{v_\ell}|= p-2$. 
Therefore, $|A_{v_0}\setminus (A_{v_\ell}\cup\{y_0\})|=1$. Assume that $\{{\hat v}\}=A_{v_0}\setminus (A_{v_\ell}\cup\{y_0\})$. 

Let $x_{\ell}, x'_{\ell}$ be two nonadjacent vertices in $S_{\ell}$ that are adjacent to $v_{\ell}$.
The vertex ${\hat v}$ is one of $x_{\ell}, x'_{\ell}$; otherwise each vertex $w$ has $p+1$ neighbors in $S_{\ell}\cup\{v_\ell\}$, a contradiction with Claim~\ref{imp}~(d).
Assume that ${\hat v}=x_{\ell}$.

If $y_0$ is adjacent to $v_\ell$, then $H[A_{v_{\ell}}\cup\{x_{\ell}, y_0\}]$ is a copy of $K_{p+1}$ a contradiction.
Therefore, we may suppose  that  $y_0$ is not adjacent to $v_\ell$.

Since $S_{\ell}$ is a maximum $K_p$-free subset of $H$, $v_{\ell}\in B_{y_{\ell-1}}$,
 and the number of $K_q$ of $H[\overline{S_\ell}\setminus \{v_{\ell}\}]$ is less than that of $H[\overline{S_0}]$, 
by using Claim~\ref{imp}~(c), we obtain for each $w\in A_{v_{\ell}}\cap A_{v_0}$ 
the induced subgraph on $(\overline{S_\ell}\setminus\{v_{\ell}\})\cup\{w\}$ has a copy of $K_q$ including $w$.
Note that $w$ has $p$ neighbors $({S_\ell}\cup\{v_{\ell}\})\setminus\{w\}$ and then exactly $q-1$ neighbors in 
$(\overline{S_\ell}\setminus\{v_{\ell}\})\cup\{w\}$. Since $y_0$ is adjacent to $w$, it follows that $y_0$ lies in this copy of $K_q$ in
$H[(\overline{S_\ell}\setminus\{v_{\ell}\})\cup\{w\}]$. Therefore, we can deduce that
 $y_0$ has at least $q-1$ neighbors in $H[(\overline{S_\ell}\cup\{w\})\setminus\{v_{\ell}\}]$ and consequently at most 
$p$ neighbors in $H[({S_\ell}\cup\{v_{\ell}\})\cup\{w\}]$.

Note that $({S_\ell}\cup\{v_{\ell}\})\setminus\{w\}$ is a maximum $K_p$-free subset of $H$.
Therefore, the  induced subgraph $H[({S_\ell}\cup\{v_\ell, y_0\})\setminus \{w\}]$ contains a copy of $K_p$ including $y_0$.
Then, the vertex set of this copy of $K_p$ must be contained  in neighbors of $y_0$ in $({S_\ell}\cup\{v_\ell, y_0\})\setminus \{w\}$.

The vertex $y_0$ in $({S_\ell}\cup\{v_\ell, y_0\})\setminus \{w\}$ is adjacent to at most $p$ neighbors, namely $v_1$, ${\hat v}=x_\ell$,
 $p-3$ vertices in $A_{v_0}\cap A_{v_\ell}\setminus\{w\}$ and possibly another vertex $z$. 
 Note that both  $v_1$ and $z$ are not adjacent to any  vertex in $A_{v_0}\cap A_{v_\ell}\setminus\{w\}$; otherwise if there is a vertex $w_0$ in
$A_{v_0}\cap A_{v_\ell}\setminus\{w\}$ is adjacent to $v_1$ (or $z$), then  $w_0$ is adjacent to $p+1$ neighbors in $S_{\ell}$
that is  $v_1$(or $z$), $x_\ell, x'_\ell$ and $p-2$ neighbors in $A_{v_\ell}$, a contradiction with  Claim~\ref{imp}~(d).

Since there exists no edge between the set $\{v_1,z\}$ and $A_{v_0}\cap A_{v_\ell}\setminus\{w\}$ and $y_0$ lies in a copy of $K_p$
in $H[({S_\ell}\cup\{v_\ell, y_0\})\setminus \{w\}]$, it follows that this copy must consist of  $y_0,v_1,x_{\ell},z$. Therefore,  we can say $p=4$ and 
surely there exists $p$th neighbor of $y_0$ in $({S_\ell}\cup\{v_\ell, y_0\})\setminus\{w\}$, that is $z$.

We have $y_0$ is in a copy of $K_q$ in $H[(\overline{S_\ell}\cup\{w\})\setminus\{v_{\ell}\}]$ 
and the number of copies $K_q$ in $H[\overline{S_\ell}\setminus \{v_{\ell}\}]$ is
equal to the number of copies $K_q$ in $H[(\overline{S_\ell}\cup\{w\})\setminus\{v_{\ell},y_0\}]$. Then 
the number of copies $K_q$ in $H[(\overline{S_\ell}\cup\{w\})\setminus\{v_{\ell},y_0\}]$ is less than that of $H[\overline{S_0}]$.
By using Claim~\ref{imp}~(c), we obtain $x_\ell$ lies in a copy of $K_q$ in $H[(\overline{S_\ell}\cup\{w, x_\ell\})\setminus\{v_{\ell},y_0\}]$.
Then $x_\ell$ has $q-1$ neighbors in $(\overline{S_\ell}\cup\{w, x_\ell\})\setminus\{v_{\ell},y_0\}$.
The vertex $x_\ell$ has $p+1$ neighbors in 
$({S_\ell}\cup\{v_\ell, y_0\})\setminus \{w, x_{\ell}\}$, namely $y_0,v_1,z$ and $p-2$ neighbors in $A_{v_\ell}\setminus\{w\}$. Therefore,
$\deg(x_{\ell})\geq p+q$, a contradiction.
\end{proof}

%#######################%#######################%#######################%#######################%% #######################%#######################
%#######################%#######################%#######################%#######################%#######################

Now we are in position to complete the proof of Theorem~\ref{mth1}
\begin{claim}
If $| A_{v_0}\cap A_{v_{\ell}}|= p-2$, then $K_{p+1}\subset H$.
\end{claim}	
\begin{proof}
Assume that  $|A_{v_0}|=|A_{v_\ell}|=p$ and  $| A_{v_0}\cap A_{v_{\ell}}|\geq p-2$.
Assume that $w\in A_{v_0}\cap A_{v_{\ell}}$. Therefore, $w$ has $p$ neighbors in $S_{\ell}\cup\{v_{\ell}\}$ and 
by Claim~\ref{imp}~{(c)}, it follows that $w$ lies in a copy of $K_q$ in $H[({\overline {S_{\ell}}}\setminus\{v_{\ell}\})\cup\{w\}]$.
Since $w$ has exactly $q-1$ neighbors in $({\overline {S_{\ell}}}\setminus\{v_{\ell}\})\cup\{w\}$ and $w$ is adjacent to $y_0$, we can deduce that
$y_0$ must be in this copy of $K_q$ in  $H[({\overline {S_{\ell}}}\setminus\{v_{\ell}\})\cup\{w\}]$.
Therefore, $y_0$ has at least $q-1$ neighbors in  $({\overline {S_{\ell}}}\setminus\{v_{\ell}\})\cup\{w\}$ and consequently $y_0$
 has at most $p$ neighbors in $({ {S_{\ell}}}\cup\{v_{\ell}\})\setminus\{w\}$.

 Since $H[(S_{\ell}\cup\{v_{\ell}\})\setminus\{w\}]$  is a maximum $K_p$-free subset in $H$,
 it follows that $H[(S_{\ell}\cup\{v_{\ell},y_0\})\setminus\{w\}]$ contains a copy of $K_p$ including $y_0$.

We have $|A_{v_\ell}\setminus A_{v_0}|=2$. Assume that $A_{v_\ell}\setminus A_{v_0}=\{v_{\ell},z\}$. If $y_0$ is adjacent to both vertices in $A_{v_\ell}\setminus A_{v_0}$, then $H[A_{v_\ell}\cup\{y_0\}]$ is a copy of $K_{p+1}$. So, we may assume that $y_0$ is adjacent to at most one of $v_{\ell}$ and $z$.

Assume that $y_0$ is adjacent to $v_{\ell}$ and is not adjacent to $z$. Let $\hat{v}$ be the only vertex of $A_{v_0}\setminus( A_{v_\ell}\cup\{y_0\})$. If $\hat{v}$ is adjacent to $v_{\ell}$, then the induced subgraph on $A_{v_0}\cup\{v_{\ell}\}$ is a copy of $K_{p+1}$, a contradiction. Therefore, we may assume that $\hat{v}$ is not adjacent to $v_{\ell}$.

 Since $y_0$ is adjacent to both $\hat{v}$ and $v_{\ell}$ and they are not adjacent to each other, the copy of $K_p$ in $H[(S_{\ell}\cup\{v_{\ell},y_0\})\setminus\{w\}]$ including $y_0$, can
 contain only one of  $\hat{v}$ and $v_{\ell}$ and 
 so it must contain all other  $p-2$ neighbors of $y_0$ in $(S_{\ell}\cup\{v_{\ell},y_0\})\setminus\{w\}$, that is,
 all $p-3$ vertices in $A_{v_0}\cap A_{v_{\ell}}\setminus\{w\}$ and $v_1$.
  Consider a vertex $w'\in (A_{v_\ell}\cap A_{v_0})\setminus\{w\}$ because $p\geq 4$ 
the existence of this vertex is guaranteed. The vertex $w'\in (A_{v_\ell}\cap A_{v_0})\setminus\{w\}$
has $p$ neighbors in $S_{\ell}\cup\{v_{\ell}\}$ and all of these $p$ neighbors are in $A_{v_0}\cup A_{v_{\ell}}\setminus\{y_0\}$.
  Since $v_1$ is not in  $A_{v_0}\cup A_{v_{\ell}}$, $w'$
  is not adjacent $v_1$, a contradiction.
If $y_0$ is not adjacent $v_{\ell}$ and is  adjacent to $z$, then by the same argument we reach a contradiction.

 So suppose that $y_0$ is adjacent to neither $v_{\ell}$ nor  $z$. 
The vertex $y_0$ has at most $p$ neighbors in $(S_{\ell}\cup\{v_{\ell}\})\setminus\{w\}$.
 It is adjacent to $p-3$ neighbors in $(A_{v_\ell}\cap A_{v_0})\setminus\{w\}$, $v_1$, $\hat{v}$, and possibly another vertex in $(S_{\ell}\cup\{v_{\ell}\})\setminus\{w\}$, 
 say $u$.
Each vertex $w'$ in $(A_{v_0}\cap A_{v_\ell})\setminus\{w\}$ is not adjacent $v_1$ and  is not adjacent to $u$ because by using Claim~\ref{imp}~{(d)},  $w'$ has 
exactly $p$ neighbors in $S_{\ell}\cup\{v_{\ell}\}$ and consequently  has $p$ neighbors in
$(S_{\ell}\cup\{v_{\ell},y_0\})\setminus\{w\}$, namely $y_0$, $\hat v$, all $p-2$ vertices in $A_{v_{\ell}}\setminus\{w,w'\}$.
Since $y_0$ has at most $p$ neighbors in $(S_{\ell}\cup\{v_{\ell},y_0\})\setminus\{w\}$,  the copy of $K_p$ in $H[(S_{\ell}\cup\{v_{\ell},y_0\})\setminus\{w\}]$ including 
$y_0$, must contain at least one of 
$v_1$ and $u$, so it cannot contain any vertex of $A_{v_{\ell}}\setminus\{w,w'\}$.
Therefore, this copy of $K_p$ must consist of $y_0,v_1,\hat{v}, u$. Then, 
  we can say $p=4$ and 
surely there exists $p$th neighbor of $y_0$ in $({S_\ell}\cup\{v_\ell, y_0\})\setminus\{w\}$, that is $u$.

Consider $\hat{v}$. It is adjacent to $p (=4)$ vertices in $(S_{\ell}\cup\{v_{\ell},y_0\})\setminus\{w\}$ and so has at most $q-1$ neighbors in 
 $({\overline {S_{\ell}}}\setminus\{v_{\ell},y_0\})\cup\{w\}$. Since $\hat v$ is in this copy of $K_4$ with vertices $y_0,v_1,\hat{v}, u$, 
 we have $(S_{\ell}\cup\{v_{\ell},y_0\})\setminus\{w,{\hat v}\}$
 is $K_4$-free with maximum size, therefore ${\hat v}$ must lie in a unique copy of $K_q$ including $w$ and all of its $q-2$ neighbors in
  $({\overline {S_{\ell}}}\setminus\{v_{\ell},y_0\})$. Therefore, $y_0,{\hat v}$, $p-2$ vertices in $A_{v_0}\cap A_{v_\ell}$, and $q-2$ neighbors of $y_0$ in
  $({\overline {S_{\ell}}}\setminus\{v_{\ell},y_0\})$ form a copy of $K_{p+q-2}$, a contradiction.

\end{proof}	
%========================================================================================
		%========================================================================================
		%========================================================================================
		%========================================================================================
\begin{claim}
If $| A_{v_0}\cap A_{v_{\ell}}|= p-1$, then $K_{p+1}\subset H$.
\end{claim}	
%========================================================================================
%========================================================================================
%========================================================================================
%========================================================================================
		
\begin{proof}
From  $| A_{v_0}| =|A_{v_{\ell}}|= p$ and $| A_{v_0}\cap A_{v_{\ell}}|= p-1$, we have
$A_{v_\ell}\setminus A_{v_{0}}=\{v_{\ell}\}$, $A_{v_0}\setminus A_{v_{\ell}}=\{y_{0}\}$. 
We may assume that $y_0$ is not adjacent to $v_\ell$, as otherwise the induced subgraph on $A_{v_\ell}\cup\{y_0\}$ would be isomorphic to $K_{p+1}$.
 It is worth noting that each vertex $w\in A_{v_0}\cap A_{v_\ell}$ has at least $p-1$ and 
 by Claim~\ref{imp}~(d) at most $p$ neighbors in $S_\ell\cup\{v_\ell\}$. To complete proof, we consider the following  two cases.

 \item[\bf Case 1:] 
   There is a vertex $w\in A_{v_0}\cap A_{v_{\ell}} $ such that 
$w$ has exactly $p$ neighbors in $S_{\ell}\cup\{v_{\ell}\}$.

Note that $v_1$ is adjacent at most one vertex of $A_{v_0}\setminus \{y_0\}$; otherwise assume that $v_1$ is adjacent 
two vertices $w_0, w'_0$ of $A_{v_0}\setminus \{y_0\}$.
Since  $A_{v_0}\setminus \{y_0\}\subseteq S_1$, $v_1$ has at most $p$ neighbors in $S_1$, and  $w_0, w'_0$ are adjacent, at least one of $w_0, w'_0$
 must be in $A_{v_1}$, which contradicts with $A_{v_0}\cap A_{v_1}=\varnothing$.

 By Claim~\ref{imp}~(c),  $H[({\overline {S_{\ell}}}\setminus\{v_{\ell}\})\cup\{w\})]$ contains a copy of $K_q$ that includes $w$. 
 As $w$ has exactly $p$ neighbors in $S_\ell\cup\{v_\ell\}$, it has $q-1$ neighbors in ${\overline {S_{\ell}}}\setminus\{v_{\ell}\}$. 
 Thus, since $w$ is adjacent to $y_0$ and $y_0$ is in ${\overline {S_{\ell}}}\setminus\{v_{\ell}\}$, the copy of $K_q$ must contain $y_0$. 
 Therefore, $y_0$ has at least $q-1$ neighbors in $({\overline {S_{\ell}}}\setminus\{v_{\ell}\})\cup\{w\}$ and at most $p$ neighbors in $(S_{\ell}\cup\{v_{\ell}\})\setminus\{w\}$.

For the vertex $w$, since $S_{\ell}$ is the maximum $K_p$-free subset in $H$, $H[({{S_{\ell}}}\cup\{v_{\ell},y_0\})\setminus\{w\}]$ contains a copy of $K_p$ that includes $y_0$. If this copy of $K_p$ contains $v_1$, then it can contain at most one vertex of $A_{v_0}\cap A_{v_{\ell}}$, say $w_0$. Since $p\geq 4$, this copy of $K_p$ in $H[({{S_{\ell}}}\cup\{v_{\ell}, y_0\})\setminus\{w\}]$ must contain at least one vertex $u$ besides $y_0$, $v_1$, and $w_0$. Consequently, in $S_{\ell}\cup\{v_{\ell}\}$, $w_0$ is adjacent to $v_1$, $u$, and $p-1$ neighbors in $A_{v_{\ell}}$, which means that $w_0$ has $p+1$ neighbors in $S_{\ell}$, a contradiction with Claim~\ref{imp}~(d).

If this copy of $K_p$ does not contain $v_1$, then it must contain all $p-2$ vertices in $A_{v_0}\cap A_{v_{\ell}}\setminus\{w\}$, 
as well as a vertex $u$ in $({{S_{\ell}}}\cup\{v_{\ell}\})\setminus\{w\}$ and outside of $A_{v_0}\cap A_{v_\ell}$.
 As a consequence, $y_0$ has exactly $p$ neighbors in $(S_{\ell}\cup\{v_{\ell}, y_0\})\setminus\{w\}$ 
 and thus, $y_0$ has exactly $q-1$ neighbors in $({\overline {S_{\ell}}}\setminus\{v_{\ell}, y_0\})\cup\{w\}$. 
 Therefore, $y_0$ has exactly $q-2$ neighbors in ${\overline {S_{\ell}}}\setminus\{v_{\ell}\}$.
  Moreover, each vertex $w'\in A_{v_0}\cap A_{v_{\ell}}\setminus\{w\}$ is adjacent to $u$ in $S_{\ell}\cup\{v_{\ell}\}$,
   and thus, $w'$ has $p$ neighbors in $S_{\ell}\cup\{v_{\ell}\}$, namely all $p-1$ vertices of $A_{v_{\ell}}$ and $u$. 
   By Claim~\ref{imp}~(c), $H[({\overline {S_{\ell}}}\setminus\{v_{\ell})\cup\{w'\}\} $
contains a copy of $K_q$ that includes $w'$. 
   Since $w'$ is adjacent to $y_0$ and $y_0$ is in ${\overline {S_{\ell}}}\setminus\{v_{\ell}\}$, 
   the copy of $K_q$ must contain $y_0$ and its neighbors in ${\overline {S_{\ell}}}\setminus\{v_{\ell}\}$. 
   Therefore, $H[A_{v_0}\cup (N(y_0)\cap \overline{ S_{\ell}})]$ is isomorphic to $K_{p+q-2}$.

 \item[\bf Case 2:] Each vertex $w\in A_{v_0}\cap A_{v_{\ell}} $ has
 exactly $p-1$ neighbors in $S_{\ell}\cup\{v_{\ell}\}$.
 
  Fix a  vertex $w\in A_{v_0}\cap A_{v_{\ell}}$. As  $S_{\ell}$  is maximum  $K_p$-free subset in $H$, we have $H[({ {S_{\ell}}}\cup\{v_{\ell},y_0\})\setminus\{w\}]$ 
 has a copy of $K_p$ including $y_0$.  All $p-1$ neighbors of each vertex of $A_{v_0}\cap A_{v_{\ell}}$ in  $H[S_{\ell}\cup\{v_\ell\}]$ lie in $A_{v_{\ell}}$. Therefore,
 each vertex of $A_{v_0}\cap A_{v_{\ell}}\setminus\{w\}$ has $p-2$ neighbors $(S_{\ell}\cup\{v_\ell\})\setminus\{w\}$ including $v_{\ell}$. Therefore,
 this copy of $K_p$ including $y_0$ does not contains any vertex of $A_{v_0}\cap A_{v_{\ell}}\setminus\{w\}$.
Since $p\geq 4$, $y_0$ has at least three neighbors in this copy of $K_p$ which are not in  $A_{v_0}\setminus\{y_0\}$. 
Therefore, $y_0$ has $p+2$ neighbors in $S_{\ell}$ and at most $q-3$ neighbors in $\overline{S_{\ell}}$.

Since $y_0$ lies in a copy of $K_q$ in  $\overline{S_1}$, so it has at least  $q-1$ neighbors in $\overline{S_0}\setminus\{v_0\}$ and 
it has at most $p$ neighbors  $S_0\cup\{v_0\}$. But $y_0$ has $p+2$ in $S_{\ell}$.
So there are at least three neighbors of $y_0$ in $S_{\ell}$ and out of 
$A_{v_0}\cap A_{v_{\ell}}$. Assume that one of them is $v_1$ and  the two other are $z,z'$. 
Then $\{v_1,z,z'\}\subseteq S_{\ell}$. Assume that $\iota$ is a smallest index for which
$\{v_1,z,z'\}\subseteq S_{\iota}$. Since $\iota$ is a smallest index for which
$\{v_1,z,z'\}\subseteq S_{\iota}$, we may assume that $z'={v_{\iota-1}}$.
Since $z'\neq v_1$, we have $\iota-1>1$.
Since $\{v_1,z\}\subseteq S_{\iota-1}$ and by Claim~\ref{0p1}~(a) we have $A_{v_0}\setminus\{y_0\}\subseteq S_{\iota-1}$, 
we deduce that  $y_0$ has at least $p+1$ neighbors in $S_{\iota-1}$ and 
at most $q-2$ neighbors in $\overline{S_{\iota-1}}$. Therefore,
the vertex $y_0$  does not lie any copy of $K_q$ in $H[\overline{S_{\iota-1}}]$.
The vertex  $z'={v_{\iota-1}}$ lies in at least one copy of $K_q$ in $H[\overline{S_{\iota-1}}]$. This copy does not contain $y_0$.
Therefore, $z'={v_{\iota-1}}$ has at least  $q$ neighbors in $\overline{S_{\iota-1}}$ and consequently at most $p-1$ neighbors in ${S_{\iota-1}}$.
 Since  $S_{\iota-1}$  is maximum  $K_p$-free subset of $H$, we have $H[S_{\iota-1}\cup\{v_{\iota-1}\}]$ must contain a copy of $K_p$.
 Since  $z'={v_{\iota-1}}$ has at most $p-1$ neighbors in ${S_{\iota-1}}$, it must lie in a unique copy of $K_p$ in $H[S_{\iota-1}\cup\{v_{\iota-1}\}]$ and $|A_{v_{\iota-1}}|=p$.
This copy must contain $v_1$. Therefore, $v_1\in A_{v_1}\cap A_{v_{\iota-1}}$ and consequently $A_{v_1}\cap A_{v_{\iota-1}}\neq \varnothing$, a contradiction.

\end{proof}

\end{proof} 
	%=========================================================================	
	%========================================================================================
	%=========================================================================	
	\begin{proof}[Proof of Corollary~\ref{cor1}] 
Since $\sum_{i=1}^k p_i=\Delta(H)-1+k$, we have $\Delta(H)-1+k\geq 11+2(k-3)$ and consequently $\Delta(H)\geq k+6\geq 9$.
We set $p=p_1$ and $q=\sum_{i=2}^kp_i-(k-2)$. We have $\omega(H)=p$ and  $\Delta(H)+1-p=q$
and $q\geq 7+2(k-3)-(k-2)=3+k\geq 6$.
 By using Theorem~\ref{mth1},  there exists  $V_1\subseteq V(H)$ such that $V_1$ is a maximum $K_p$-free subset of $H$,
		and $H[V\setminus V_1] $ is $K_q$-free.	Now consider the graph  $H[V\setminus V_1]$. Its clique number is at most $q-1$.
		Since is $V_1$ a maximum $K_p$-free subset of $H$, the maximum degree of $H[V\setminus V_1]$ at most  $\Delta(H)-(p-1)=q$.
		By similar argument to in the proof of Theorem~\ref{mainthm} we may assume that the maximum degree of $H[V\setminus V_1]$ is equal to $q$.
				Note that $\sum_{i=2}^kp_i=q-1+(k-1)$ and $p_2+p_3\geq 7$.
				By applying Theorem~\ref{mainthm} for $H[V\setminus V_1]$, we get that there exists a partition of $V\setminus V_1$
				into $V_2,V_3,\ldots, V_k$
		such that for each $2\leq i\leq k$, $H[V_i]$ is $K_{p_i}$-free.
	
	\end{proof} 
	
	%========================================================================================
	%=========================================================================	
	%========================================================================================

	%========================================================================================
	
	\section{Concluding remarks }
	In this section, we pose three  questions that  remain unsolved as parts of Question~\ref{genBK}.
	\begin{question} \label{qu3}
		\item[(a)]  Let $H$ be a graph  with $\omega(H)\leq \Delta(H)-1$. 
		Assume  that  $p_1\geq p_2\geq\cdots\geq p_k\geq 2$  are  $k$ positive integers and  $\sum_{i=1}^k p_i=\Delta(H)-1+k$. If  $p_1+p_2\leq 6$ and 
		$p_1\geq 3$, then
		does there exist a partition of  $V(H)$ into $V_1,V_2,\ldots, V_k$
		so that $H[V_i]$ is $K_{p_i}$-free for each $1\leq i\leq k$?
		\item[(b)] Suppose that $H$ is a graph with $\Delta(H)=5$ and clique number $\omega(H)\leq 4$. 
		Does there exist a partition of $V(H)$ into $V_1$ and $V_2$
		such that both of  $H[V_1]$ and $H[V_2]$ are $K_3$-free?
	 		\item[(c)] Suppose that $H$ is a graph with $\Delta(H)= 6$ and clique number $\omega(H)\leq 5$. 
		Does there exist a partition of $V(H)$ into $V_1,V_2, V_3$
		such that $H[V_1]$ is $K_4$-free and both $H[V_2]$ and $H[V_3]$ are $K_2$-free?
	\end{question}
	Note that in Remark~\ref{rmk1} we give some negative answers for Question~\ref{qu3}~(a), for the following cases: 
	\begin{itemize}
	\item $k=2$, $p_1=3$, and $p_2=2$ (see Remark~\ref{rmk1}~(a)).
	\item $k=2$, $p_1=4$, $p_2=2$, (see Remark~\ref{rmk1}~(b)).
	\item $k=3$, $p_1=3$, $p_2=2$, and $p_3=2$ (see Remark~\ref{rmk1}~(c)).
	\end{itemize}

	   It worth mentioning that by using Theorem~\ref{M.th2},  it can be shown that if Question~\ref{qu3}~(b) has an affirmative answer, then 
	  an affirmative answer to Question~\ref{qu3}~(a) can be obtained when $p_1=p_2=3$. 
	  Similarly,  if  Question~\ref{qu3}~(c) has an affirmative answer, 
	 from Theorem~\ref{M.th2} it follows that  
	   an affirmative answer to Question~\ref{qu3}~(a) can be obtained when $p_1=4$ and $k\geq 3$.

	%%%%%%%%%%%%%%%%%%%%%%%%%%%%%%%%%%%%%%%%%
	%%%%%%%%%%%%%%%%%%%%%%%%%%%%%%%%%%%%%%%
	%%%%%%%%%%%%%%%%%%%%%%%%%%%%%%%%%%%%%%%%%%%%%%%%%%%%%%%%%%%%%%%%%%%%%%%%%%%%%%%%	
%%%%%%%%%%%%%%%%%%%%%%%%%%%%%%%%%%%%%%%%%%%%%%%%%%%%%%%%%%%%%%%%%%%%%%%%%%%%%%%%	
%%%%%%%%%%%%%%%%%%%%%%%%%%%%%%%%%%%%%%%%%%%%%%%%%%%%%%%%%%%%%%%%%%%%%%%%%%%%%%%%	
	%%%%%%%%%%%%%%%%%%%%%%%%%%%%%%%%%%%%%%%%%
	%%%%%%%%%%%%%%%%%%%%%%%%%%%%%%%%%%%%%%%%%
	%%%%%%%%%%%%%%%%%%%%%%%%%%%%%%%%%%%%%%%%%
	%%%%%%%%%%%%%%%%%%%%%%%%%%%%%%%%%%%%%%%%%
	%%%%%%%%%%%%%%%%%%%%%%%%%%%%%%%%%%%%%%%%%
	%%%%%%%%%%%%%%%%%%%%%%%%%%%%%%%%%%%%%%%%%
	%%%%%%%%%%%%%%%%%%%%%%%%%%%%%%%%%%%%%%%%%
	%%%%%%%%%%%%%%%%%%%%%%%%%%%%%%%%%%%%%%%%%
	%%%%%%%%%%%%%%%%%%%%%%%%%%%%%%%%%%%%%%%%%
	%%%%%%%%%%%%%%%%%%%%%%%%%%%%%%%%%%%%%%%%%
	%%%%%%%%%%%%%%%%%%%%%%%%%%%%%%%%%%%%%%%%%
	%%%%%%%%%%%%%%%%%%%%%%%%%%%%%%%%%%%%%%%%%
	%%%%%%%%%%%%%%%%%%%%%%%%%%%%%%%%%%%%%%%%%
	%========================================================================================
	%%%%%%%%%%%%%%%%%%%%%%%%%%%%%%%%%%%%%%%%%
	%%%%%%%%%%%%%%%%%%%%%%%%%%%%%%%%%%%%%%%%%
	%%%%%%%%%%%%%%%%%%%%%%%%%%%%%%%%%%%%%%%%%
	%%%%%%%%%%%%%%%%%%%%%%%%%%%%%%%%%%%%%%%%%
	%%%%%%%%%%%%%%%%%%%%%%%%%%%%%%%%%%%%%%%%%
	\bibliographystyle{plain}

\end{document}